\documentclass[12pt,a4paper,reqno,intlimits,sumlimits]{amsart}
\usepackage[numbers]{natbib}
\usepackage{amsmath}
\usepackage{amssymb}
\usepackage[mathscr]{euscript}
\usepackage{enumerate}
\usepackage{xspace}
\usepackage{color}
\usepackage{enumerate}
\usepackage{enumitem}
\usepackage{amsthm}
\usepackage{mathptmx}

\begin{document}

\theoremstyle{plain}
\newtheorem{C}{Convention}
\newtheorem{SA}{Standing Assumption}
\newtheorem{theorem}{Theorem}[section]
\newtheorem{condition}{Conditions}[section]
\newtheorem{lemma}[theorem]{Lemma}
\newtheorem{proposition}[theorem]{Proposition}
\newtheorem{corollary}[theorem]{Corollary}
\newtheorem{claim}[theorem]{Claim}
\newtheorem{definition}[theorem]{Definition}
\newtheorem{Ass}[theorem]{Assumption}
\newcommand{\q}{{\mathbf{Q}}}
\theoremstyle{definition}
\newtheorem{remark}[theorem]{Remark}
\newtheorem{note}[theorem]{Note}
\newtheorem{example}[theorem]{Example}
\newtheorem{assumption}[theorem]{Assumption}
\newtheorem*{notation}{Notation}
\newtheorem*{assuL}{Assumption ($\mathbb{L}$)}
\newtheorem*{assuAC}{Assumption ($\mathbb{AC}$)}
\newtheorem*{assuEM}{Assumption ($\mathbb{EM}$)}
\newtheorem*{assuES}{Assumption ($\mathbb{ES}$)}
\newtheorem*{assuM}{Assumption ($\mathbb{M}$)}
\newtheorem*{assuMM}{Assumption ($\mathbb{M}'$)}
\newtheorem*{assuL1}{Assumption ($\mathbb{L}1$)}
\newtheorem*{assuL2}{Assumption ($\mathbb{L}2$)}
\newtheorem*{assuL3}{Assumption ($\mathbb{L}3$)}
\newtheorem{charact}[theorem]{Characterization}
\newcommand{\notiz}{\textup} 
\renewenvironment{proof}{{\parindent 0pt \it{ Proof:}}}{\mbox{}\hfill\mbox{$\Box\hspace{-0.5mm}$}\vskip 16pt}
\newenvironment{proofthm}[1]{{\parindent 0pt \it Proof of Theorem #1:}}{\mbox{}\hfill\mbox{$\Box\hspace{-0.5mm}$}\vskip 16pt}
\newenvironment{prooflemma}[1]{{\parindent 0pt \it Proof of Lemma #1:}}{\mbox{}\hfill\mbox{$\Box\hspace{-0.5mm}$}\vskip 16pt}
\newenvironment{proofcor}[1]{{\parindent 0pt \it Proof of Corollary #1:}}{\mbox{}\hfill\mbox{$\Box\hspace{-0.5mm}$}\vskip 16pt}
\newenvironment{proofprop}[1]{{\parindent 0pt \it Proof of Proposition #1:}}{\mbox{}\hfill\mbox{$\Box\hspace{-0.5mm}$}\vskip 16pt}

\newcommand{\Law}{\ensuremath{\mathop{\mathrm{Law}}}}
\newcommand{\loc}{{\mathrm{loc}}}
\newcommand{\Log}{\ensuremath{\mathop{\mathcal{L}\mathrm{og}}}}
\newcommand{\Meixner}{\ensuremath{\mathop{\mathrm{Meixner}}}}
\newcommand{\of}{[\hspace{-0.06cm}[}
\newcommand{\gs}{]\hspace{-0.06cm}]}

\let\MID\mid
\renewcommand{\mid}{|}

\let\SETMINUS\setminus
\renewcommand{\setminus}{\backslash}

\def\stackrelboth#1#2#3{\mathrel{\mathop{#2}\limits^{#1}_{#3}}}

\renewcommand{\theequation}{\thesection.\arabic{equation}}
\numberwithin{equation}{section}

\newcommand\llambda{{\mathchoice
      {\lambda\mkern-4.5mu{\raisebox{.4ex}{\scriptsize$\backslash$}}}
      {\lambda\mkern-4.83mu{\raisebox{.4ex}{\scriptsize$\backslash$}}}
      {\lambda\mkern-4.5mu{\raisebox{.2ex}{\footnotesize$\scriptscriptstyle\backslash$}}}
      {\lambda\mkern-5.0mu{\raisebox{.2ex}{\tiny$\scriptscriptstyle\backslash$}}}}}

\newcommand{\prozess}[1][L]{{\ensuremath{#1=(#1_t)_{0\le t\le T}}}\xspace}
\newcommand{\prazess}[1][L]{{\ensuremath{#1=(#1_t)_{0\le t\le T^*}}}\xspace}

\newcommand{\tr}{\operatorname{tr}}
\newcommand{\lijepoa}{{\mathcal{A}}}
\newcommand{\lijepob}{{\mathcal{B}}}
\newcommand{\lijepoc}{{\mathcal{C}}}
\newcommand{\lijepod}{{\mathcal{D}}}
\newcommand{\lijepoe}{{\mathcal{E}}}
\newcommand{\lijepof}{{\mathcal{F}}}
\newcommand{\lijepog}{{\mathcal{G}}}
\newcommand{\lijepok}{{\mathcal{K}}}
\newcommand{\lijepoo}{{\mathcal{O}}}
\newcommand{\lijepop}{{\mathcal{P}}}
\newcommand{\lijepoh}{{\mathcal{H}}}
\newcommand{\lijepom}{{\mathcal{M}}}
\newcommand{\lijepou}{{\mathcal{U}}}
\newcommand{\lijepov}{{\mathcal{V}}}
\newcommand{\lijepoy}{{\mathcal{Y}}}
\newcommand{\cF}{{\mathcal{F}}}
\newcommand{\cG}{{\mathcal{G}}}
\newcommand{\cH}{{\mathcal{H}}}
\newcommand{\cM}{{\mathcal{M}}}
\newcommand{\cD}{{\mathcal{D}}}
\newcommand{\bD}{{\mathbb{D}}}
\newcommand{\bF}{{\mathbb{F}}}
\newcommand{\bG}{{\mathbb{G}}}
\newcommand{\bH}{{\mathbb{H}}}
\newcommand{\dd}{\operatorname{d}\hspace{-0.003cm}}
\newcommand{\ddd}{\operatorname{d}}
\newcommand{\er}{{\mathbb{R}}}
\newcommand{\ce}{{\mathbb{C}}}
\newcommand{\erd}{{\mathbb{R}^{d}}}
\newcommand{\en}{{\mathbb{N}}}
\newcommand{\de}{{\mathrm{d}}}
\newcommand{\im}{{\mathrm{i}}}
\newcommand{\indik}{{\mathbf{1}}}
\newcommand{\D}{{\mathbb{D}}}
\newcommand{\E}{{\mathbf{E}}}
\newcommand{\N}{{\mathbb{N}}}
\newcommand{\Q}{{\mathbb{Q}}}
\renewcommand{\P}{{\mathbb{P}}}
\newcommand{\ud}{\operatorname{d}\!}
\newcommand{\ii}{\operatorname{i}\kern -0.8pt}
\newcommand{\Var}{\operatorname{Var }\,}
\newcommand{\dt}{\operatorname{d}\!t}   
\newcommand{\ds}{\operatorname{d}\!s}   
\newcommand{\dy}{\operatorname{d}\!y }    
\newcommand{\du}{\operatorname{d}\!u}  
\newcommand{\dv}{\operatorname{d}\!v}   
\newcommand{\dx}{\operatorname{d}\!x}   
\newcommand{\dq}{\operatorname{d}\!q}   
\newcommand{\cadlag}{c\`adl\`ag }
\newcommand{\p}{{\mathbf{P}}}
\newcommand{\F}{\mathfrak{F}}
\newcommand{\1}{\mathbf{1}}
\newcommand{\f}{\mathcal{F}^{\hspace{0.03cm}0}}
\newcommand{\lle}{\langle\hspace{-0.085cm}\langle}
\newcommand{\rre}{\rangle\hspace{-0.085cm}\rangle}
\newcommand{\llbr}{[\hspace{-0.085cm}[}
\newcommand{\rrbr}{]\hspace{-0.085cm}]}

\def\EM{\ensuremath{(\mathbb{EM})}\xspace}

\newcommand{\la}{\langle}
\newcommand{\ra}{\rangle}

\newcommand{\Norml}[1]{%
{|}\kern-.25ex{|}\kern-.25ex{|}#1{|}\kern-.25ex{|}\kern-.25ex{|}}

\title[]{Martingale Property in Terms of Semimartingale Problems} 

\author[D. Criens]{David Criens}
\address{D. Criens - Technical University of Munich, Center for Mathematics, Germany}
\email{david.criens@tum.de}
\author[K. Glau]{Kathrin Glau}
\address{K. Glau - Technical University of Munich, Center for Mathematics, Germany}
\email{kathrin.glau@tum.de}
%
%
%

\keywords{martingale, local martingale, strict local martingale, semimartingale, semimartingale problem, stochastic exponential, explosion, infinite dimensional Brownian motion
\vspace{1ex}}




%
%
%

\thanks{D. Criens - Technische Universit\"at M\"unchen, Center for Mathematics, Germany,  \texttt{david.criens@tum.de}.}
\thanks{
K. Glau - Technische Universit\"at M\"unchen, Center for Mathematics, Germany, \texttt{kathrin.glau@tum.de}.\vspace{0.5ex}
}
\thanks{\textit{Acknowledgement:} The authors thank Jean Jacod for fruitful discussions. 
}

\date{\today}
\maketitle

\frenchspacing
\pagestyle{myheadings}


\begin{abstract}
Starting from the seventies mathematicians face the question whether a non-negative local martingale is a true or a strict local martingale.
In this article we answer this question from a semimartingale perspective. We connect the martingale property to existence, uniqueness and topological properties of semimartingale problems. 
This not only leads to valuable characterizations of the martingale property, but also reveals new existence and uniqueness results for semimartingale problems.
As a case study we derive explicit conditions for the martingale property of stochastic exponentials driven by infinite-dimensional Brownian motion.
\end{abstract}



\section{True or Strict Local Martingale}
\subsection{Introduction}
The question whether a non-negative local martingale is a strict local or a true martingale is of fundamental probabilistic nature. It relates to existence of solutions to stochastic differential equations and martingale problems, as well as to
absolute continuity 
of laws.

The question has been tackled in different settings most prominently by integrability conditions of Novikov- and Kazamaki-type, c.f. \cite{J79,KS(2002b),kazamaki77,LM,Novikov73,protter,RufNK}. A series of papers further explored the one-dimensional diffusion setting that thanks to the work \cite{MU(2012)} is particularly well understood. 
Their main result relates martingality, i.e. the martingale property, of a generalized stochastic exponential driven by a solution \((X, W)\) to an SDE of the type
\begin{align}\label{SDE1}
\dd X_t = b(X_t)\dd t + a(X_t)\dd W_t
\end{align} to exit times of a solution to a modified SDE 
\begin{align}\label{SDE2}
\dd Y_t = \widetilde{b}(Y_t) \dd t + a(Y_t)\dd B_t. 
\end{align}
In particular, the result distinguishes at which end of the state space \((l, r)\) the modified solution exits.
A remarkable observation is that there are examples where the generalized stochastic exponential is a martingale while both the solution of the driving and the modified SDE exit their state space. In the case where it is the real line, this means that both SDEs may be explosive.


Already the topology of multidimensional state spaces does not allow a direct generalization of this subtle observation 
to higher dimensions. 
In finite-dimensional diffusion settings \cite{HR,Sin} connect the martingality of a non-explosive stochastic exponential driven by a solution to an SDE to existence properties of a related modified SDE.
\cite{RufSDE} starts in a possibly explosive diffusion setting and gives a condition for martingality, which is in the spirit of absolute continuity conditions for laws of semimartingales as given by \cite{JS}. 
The proof of his condition is based on the Dambis-Dubins-Schwarz theorem and hence restricted to a continuous path setting.

A condition for martingality in a jump-diffusion setting can be deduced from the main result of \cite{CFY}. 
\cite{KMK,Mayerhofer2011568} provide conditions for affine processes.
Based on an extension of probability measures \cite{kardaras2015} provide techniques to construct strict local martingales. 
\cite{BR16} study martingality from a perspective of weak convergence.

In this article we raise the question how martingality of non-negative local martingales driven by Hilbert-space valued semimartingales on stochastic intervals relates to path properties of their driving processes.
Building our framework on semimartingale problems (SMPs) in the sense of \cite{J79} allows us to connect martingality with several interesting properties of semimartingales.
To briefly describe the concept let us be given a filtered space, a process and a triplet, which serves as a candidate for semimartingale characteristics. A probability measure under which the process is a semimartingale with the candidate triplet as characteristics is called a solution to the SMP.
We particularly allow for infinite-dimensions and stochastic lifetimes. 

Let \(Z\) be a non-negative local martingale (usually) defined on a filtered probability space equipped with a solution to a given SMP.
In the diffusion setting this boils down to the assumption that \(Z\) is driven by ~\eqref{SDE1}.
We now summarize our main contributions. 

Firstly, we prove under mild topological assumptions that if \(Z\) is a martingale then a modified SMP, comparable to the modified SDE \eqref{SDE2}, has a solution, c.f. Theorem \ref{main theorem new 2}. This is not only an existence result for solutions to SMPs, but also provides a condition for strict local martingality via contradiction.

Secondly, we replace the topological assumption by a uniqueness condition on the modified SMP and assume existence of a solution.
Our uniqueness condition is related to the definition of local uniqueness used by 
\cite{JS} in their study of absolute continuity of laws. In Markovian settings it is implied by well-posedness of the SMP. If the SMP corresponds to a not necessarily Markovian SDE the uniqueness is implied by pathwise-uniqueness.
In this setting we formulate abstract convergence conditions in terms of a sequence of stopping times, c.f. Proposition \ref{main prop uni 1} and \ref{main prop uni 2}.
We use these results to formulate a boundedness condition of local type that also can be verified in cases where the conditions of Novikov- and Kazamaki-type are too strict, c.f. Section \ref{MSE}. 
In Section \ref{Martingality of Generalized Stochastic Exponentials} we investigate 
stochastic exponentials as explicit examples.

To illustrate these claims in a non-Markovian and infinite-dimensional setting let 
\[W^* := \sup_{s < \cdot} W_s\ \textup{ and }\ N := W^* \cdot W,
\] 
where \(W\) is an (infinite-dimensional) Brownian motion.
In the one-dimensinal case, we have \(\E[\exp\{(W^*)^2/2 \cdot I_t\}] \geq \E[\exp\{W^2/2 \cdot I_t\}] = \infty\), for \(t\) large enough, where \(I_t = t\) denotes the identity process. This implies that Novikov's condition is violated.
We will see in Section \ref{A Case Study - Martingality in Terms of SDEs driven by Hilbert-Space-Valued Brownian Motion} that \(Z\) is in fact a true martingale.

Thirdly, we show that under appropriate mild assumptions existence of a solution to the modified SMP implies martingality even if the original SMP has no global solution, c.f. Corollary \ref{explosion cond coro}.
This can be interpreted as a type of explosion condition as given by \cite{MU(2012)}, but now in 
an infinite-dimensional and discontinuous setting. 

Fourthly, we impose additional topological assumptions on the underlying filtered space. 
This setting traces back to ideas employed by \cite{follmer72,perkowski2015} in the context of the F\"ollmer measure.
Thanks to an extension result for probability measures these conditions guarantee uniqueness and existence properties of the modified SMP, which lead to an explosion-type condition for martingality, c.f. Section \ref{Martingality on Standard Systems}.

Martingale criteria for non-negative local martingales driven by semimartingales are also interesting from an applications point of view. 
The class of semimartingales lies at the heart of financial modeling as it plays a fundamental role for stochastic integration, c.f. \cite{bichteler1981stochastic,DellacherieMeyer78,Pro} and \cite{DS}, Chapter 7.3.
In a semimartingale setting the seminal work \cite{DelbaenSchachermayer94,Delbaen96thefundamental} relates absense of arbitrage to existence of an equivalent martingale measure (EMM). The existence of an EMM or more generally of changes of probability measures is deeply connected to the question when a candidate density process is a true martingale.
Typically candidate processes are already non-negative local martingales.
On a more practical side, changes of measures frequently lead to a substantial reduction of 
computational complexity. In mathematical finance for instance they arise as changes of numeraire and simplify the computation of option prices as 
highlighted in the influential work
\cite{GEK95}. 
In the context of utility maximization, \cite{RePEc:wsi:ijtafx:v:13:y:2010:i:03:p:459-477} connects the martingality of a local martingale with optimal terminal wealth.
In all these cases, criteria for martingality of non-negative local martingales are substantial.

Recently also strict local martingales have attracted more attention and are used to model financial bubbles, c.f. \cite{Cox2005,Jarrow2007,MAFI:MAFI394,kardaras2015,Pal20101424,Pro(2013)}. 
Examples of strict local martingales are given by \cite{Chybiryakov2007,Delbaen1995,ELY,KellerRessel2015,MU(2012),protter2015strict}. 



Martingality of local martingales is also interesting from a purely probabilistic point of view, as
it determines absolute continuity or singularity of laws. This observation was exploited in the seminal work \cite{doi:10.1137/1105027}, c.f. also \cite{JM76,KLS-LACOM1,KLS-LACOM2} for a semimartingale perspective. 
Such absolute continuity results are valuable tools to study the statistics of stochastic processes, c.f. \cite{10.2307/4616544}, and existence properties of SDEs and martingale problems, c.f. \cite{deprato,KaraShre,RY}.

Let us shortly summarize the structure of the article.
In Section \ref{First important Observations} we present simple characterizations of martingality of a non-negative local martingale.
These observations prove useful for establishing our main results. 
In Section ~\ref{Hilbert-Vlaued Semimartingales and Semimartingale Problems on Stochastic Intervals} we introduce our mathematical setting. 
The main results are given in Section \ref{section: Martingality of local martingales}. 
Finally, in Section \ref{A Case Study - Martingality in Terms of SDEs driven by Hilbert-Space-Valued Brownian Motion}, we deduce explicit martingale conditions for stochastic exponentials driven by infinite-dimensional Brownian motion.

\textbf{A Short Remark Concerning Notations:}
All non-explained notations can be found in the monograph of \cite{JS}. 
Moreover, we usually skip the terminology \emph{up to indistinguishability}. \emph{Equality of processes} as well as \emph{uniqueness of processes} should be read up to indistinguishability. 


\subsection{A Simple Description of Martingality}\label{First important Observations}
We start with a characterization of the martingality of non-negative local matingales, which is as elementary as useful.
Fatou's lemma implies that a non-negative local martingale is a supermartingale, which itself is a martingale if and only if it has constant expectation.
Therefore, the martingality of non-negative martingales is a property that only depends on the expectation. 
The following two consequences take advantage of this important observation. 

We fix a filtered probability space \((\Omega, \mathcal{F}, \F, \p)\), where \(\F = (\mathcal{F}_t)_{t \geq 0}\) is a right-continuous filtration, 
and a one-dimensional non-negative local \((\F, \p)\)-martingale \(Z\) with \(\p\)-a.s. \(Z_0=1\). 
The assumption that \(Z_0\) is \(\p\)-a.s. constant can be relaxed to the case where \(Z_0\) is a positive \(\mathcal{F}_0\)-measurable and \(\p\)-integrable random variable by replacing \(Z\) with \(Z/Z_0\).

The following lemma is in the spirit of Theorem 1.3.5 in \cite{SV}, Lemma III.3.3 in \cite{JS}, and Corollary 2.1 in \cite{BR16}. 
\begin{lemma}\label{ruf mimic}
The following is equivalent:
\begin{enumerate}
\item[\textup{(i)}]
 \(Z\) is an \((\F, \p)\)-martingale.
\item[\textup{(ii)}] There exists an \((\F, \p)\)-localization sequence \((\rho_n)_{n \in \mathbb{N}}\) of \(Z\) such that
\begin{align}\label{difficult to check}
\lim_{n \to \infty} \q_n (\rho_n> t) = 1\ \textit{ for all } t \geq 0,
\end{align}
where \(\q_n := Z_{\rho_n} \cdot \p\), i.e. \(\q_n(\dd \omega) = Z_{\rho_n}(\omega)\p(\dd \omega)\), for \(n \in \mathbb{N}\).
\item[\textup{(iii)}]
For all \((\F, \p)\)-localizing sequences \((\rho_n)_{n \in \mathbb{N}}\) of \(Z\) the convergence \eqref{difficult to check} holds.
\end{enumerate}
\end{lemma}
\begin{proof}
The implication (iii) \(\Longrightarrow\) (ii) is trivial. We show that (ii) \(\Longrightarrow\) (i) \(\Longrightarrow\) (iii).
\(Z\) is an \((\F, \p)\)-martingale if and only if for all \(t \geq 0\) we have \(\E[Z_t] = \E[Z_0] = 1\). 
Let \((\rho_n)_{n \in \mathbb{N}}\) be an arbitrary \((\F, \p)\)-localizing sequence. Due to monotone convergence, Doob's stopping theorem and \(\{\rho_n > t\} \in \mathcal{F}_{t \wedge \rho_n}\), we have
\begin{align*}
\E[Z_t] &= \E\big[Z_t \lim_{n \to \infty} \1_{\{\rho_n > t\}}\big] = \lim_{n \to \infty} \E\big[Z_{t \wedge \rho_n} \1_{\{\rho_n> t\}}\big] = \lim_{n \to \infty} \E\big[\E\big[Z_{\rho_n}|\mathcal{F}_{t \wedge \rho_n}\big]\1_{\{\rho_n> t\}}\big]
\\&= \lim_{n \to \infty} \E\big[\E\big[Z_{\rho_n} \1_{\{\rho_n > t\}}|\mathcal{F}_{t \wedge \rho_n}\big]\big]
= \lim_{n \to \infty} \q_n(\rho_n> t).
\end{align*}
This finishes the proof.
\end{proof}
If we relinquish on an equivalence statement we can relax the assumption in Lemma \ref{ruf mimic} that \((\rho_n)_{n \in \mathbb{N}}\) is a localizing sequence.
\begin{lemma}\label{neues lemma}
Assume that \((\rho_n)_{n \in \mathbb{N}}\) is a increasing sequence of \(\F\)-stopping times such that for all \(n \in \mathbb{N}\) the process \(Z^{\rho_n}\) is an uniformly integrable \((\F, \p)\)-martingale.
\begin{enumerate}
\item[\textup{(i)}]
If \eqref{difficult to check} holds for all \(t \geq 0\), then \(Z\) is an \((\F, \p)\)-martingale.
\item[\textup{(ii)}]
If 
\begin{align*}
\lim_{n \to \infty} \q_n(\rho_n = \infty) = 1,\textit{ where } \q_n = Z_{\rho_n} \cdot \p,
\end{align*}
then \(Z\) is a uniformly integrable \((\F, \p)\)-martingale.
\end{enumerate}
\end{lemma}
\begin{proof}
We first prove (i) and obtain that 
\begin{align*}
1 \geq \E[Z_t] \geq \E[Z_t \lim_{n \to \infty} \1_{\{\rho_n> t\}}] = \lim_{n \to \infty} \E[Z_t \1_{\{\rho_n > t\}}] = \lim_{n \to \infty} \q_n(\rho_n > t) = 1,
\end{align*}
by the same argumentation as in the proof of Lemma \ref{ruf mimic}. Therefore, as \(Z\) is an \((\F, \p)\)-supermartingale, \(Z\) is an \((\F, \p)\)-martingale.

We now turn to (ii). 
\(Z\) is a uniformly integrable \((\F, \p)\)-martingale if \(\E[Z_\infty] \geq 1\), where \(Z_\infty := \lim_{t \to \infty} Z_t\) exists due to the supermartingale convergences theorem. We obtain that
\begin{align*}
\E[Z_\infty] &\geq \lim_{n \to \infty} \E[Z_\infty \1_{\{\rho_n = \infty\}}] = \lim_{n \to \infty} \E[ \E[Z_\infty |\mathcal{F}_{\rho_n}] \1_{\{\rho_n = \infty\}}]
= \lim_{n \to \infty} \q_n(\rho_n = \infty) = 1,
\end{align*}
thanks to \(\{\rho_n = \infty\} \in \mathcal{F}_{\rho_n}\) and Doob's stopping theorem. 
This finishes the proof.
\end{proof}
Our main results introduce more structure for these conditions and thereby increase their applicability.
In order to identify a promising structure, let us observe the following: 
If there exists a probability measure \(\q\) such that 
\begin{align*}
\q = \q_n = Z_{\rho_n}\cdot \p \textup{ on } \mathcal{F}_{\rho_n},\textup{ for all } n \in \mathbb{N}, 
\end{align*}
then \eqref{difficult to check} boils down to 
\begin{align}\label{cond q}
\lim_{n \to \infty} \q(\rho_n > t) = 1,\textup{ for all } t \geq 0.
\end{align}
Moreover, if \eqref{cond q} holds and \(\p\)-a.s. \(\rho_n \uparrow \infty\) as in Lemma \ref{ruf mimic}, then Lemma III.3.3 in \cite{JS} yields that \(\q \ll_\textup{loc} \p\) with density process \(Z\).
To use this observation and identify a structure of \(\q\), let us introduce an additional player, an \((\F, \p)\)-semimartingale \(X\).
It is well-known that \(\q \ll_\textup{loc}\p\) implies that \(X\) is also an \((\F, \q)\)-semimartingale and its characteristics are related to \(X\) and \(Z\) via Girsanov's theorem.
Therefore \(\q\) is a probability measure under which \(X\) is a semimartingale with known characteristics.
This observation leads us to the concept of semimartingale problems. 
\section{Hilbert-space-valued Semimartingales and \\ Semimartingale Problems on Stochastic Intervals}
\label{Hilbert-Vlaued Semimartingales and Semimartingale Problems on Stochastic Intervals}
Throughout this section we fix a filtered probability space \((\Omega, \mathcal{F}, \F, \p)\), where \(\F = (\mathcal{F}_t)_{t \geq 0}\) is a right-continuous filtration.
For the first part of Section
\ref{Processes on Stochastic Intervals} the probability measure \(\p\) does not play a ~role.
\subsection{Processes on Stochastic Intervals}
\label{Processes on Stochastic Intervals}
In this first section 
we introduce the concept of stochastic processes on stochastic intervals, which build the mathematical background to study possibly explosive processes. 
We partially follow \cite{HWY,J79, JS}. 
Related topics were also studied by \cite{emery1982, Schwartz1981,sharpe}. 
We fix a 
Polish space \(E\), which serves as the \emph{state space}. 
Recall that for two \(\F\)-stopping times \(\tau\) and \(\rho\) we set
\begin{align*}
\of \tau, \rho\gs := \big\{(\omega, t) \in \Omega \times \mathbb{R}^+ : \tau(\omega) \leq t \leq \rho(\omega)\big\},
\end{align*}
and \(\of \tau, \rho\of, \gs \tau, \rho\gs, \gs \tau, \rho\of\) analogeously. Moreover, we denote \(\of \tau, \tau\gs =: \of \tau\gs\).
\begin{definition}
We call a random set \(A\) a \emph{set of interval type}, if there exists an increasing sequence of \(\F\)-stopping times \((\tau_n)_{n \in \mathbb{N}}\) such that 
\begin{align*}
A = \bigcup_{n \in \mathbb{N}} \of 0, \tau_n\gs.
\end{align*}
\end{definition}
Sets of interval type are \(\F\)-predictable. 
With a little abuse of terminology, we call the sequence \((\tau_n)_{n \in \mathbb{N}}\) \emph{announcing sequence for \(A\)}.
The set of all random sets of interval type is denoted by \(\mathcal{I}(\F)\). 
Let \(\mathcal{C}\) be a class of \(E\)-valued processes. 
For \(A \in \mathcal{I}(\F)\) we set 
\begin{align*}
\mathcal{C}^{A} &:= \big\{ X : A \to\ E : \exists \textup{ announcing sequence } (\tau_n)_{n \in \mathbb{N}} \textup{ for } A \textup{ s.th. } 
X^{\tau_n} \in \mathcal{C}\ \forall n \in \mathbb{N}\big\}.
\end{align*}
Since \((\omega, t \wedge \tau_n (\omega)) \in A\) for all \((\omega, t) \in \of 0, \infty\of\), we have \(X^{\tau_n} : \of 0, \infty\of\ \to E\) for all \(X \in \mathcal{C}^A\). 
\begin{proposition}\label{subinterval semimartingale}
Let \(\mathcal{C}\) be stable under stopping and localization, \(X \in \mathcal{C}^{A}\), and \(\rho\) an \(\F\)-stopping time such that 
\(\of 0, \rho\gs\subseteq A\). Then \(X^{\rho}\in \mathcal{C}\).
\end{proposition}
\begin{proof}
The claim follows identically to \cite{HWY}, Theorem 8.20. 
\end{proof}
\vspace{-0.3cm}
Examples of classes which are stable under stopping and localization are the class of \(E\)-valued local \((\F, \p)\)-martingales, denoted by \(\mathcal{M}_\textup{loc}(E, \F, \p)\), and the class of \(E\)-valued \((\F, \p)\)-semimartingales, denoted by \(\mathcal{S}(E, \F, \p)\), where \(E\) is a real separable Hilbert space. 
More generally, if \(\mathcal{C}\) is stable under stopping, then its localized class \(\mathcal{C}_\textup{loc}\) is stable under stopping and localization, c.f. \cite{JS}, Lemma I.1.35.
We obtain two obvious consequences of Proposition \ref{subinterval semimartingale}:
\begin{corollary}\label{ordinary semimartingale}
Let \(\mathcal{C}\) be stable under stopping and localization. 
\begin{enumerate}
\item[\textup{(i)}]
\(Y\in\mathcal{C}^{A}\) if and only if for all \(\F\)-stopping times \(\rho\) such that \(\of 0, \rho\gs \subseteq A\), \(Y^{\rho}\in \mathcal{C}\).
\item[\textup{(ii)}]
\(\mathcal{C}^{\of 0, \infty\of} = \mathcal{C}\).
\end{enumerate}
\end{corollary}
To construct non-negative local martingales from semimartingale on sets of interval type the following extension due to \cite{J79} proves itself useful in the Sections \ref{MSE} and \ref{Martingality of Generalized Stochastic Exponentials}.
For \(X\in \mathcal{C}^A\), where \(\mathcal{C}\) is a class of \([- \infty, \infty]\)-valued processes, we define the extension \(\widetilde{X}\) by
\begin{align}\label{extension}
\widetilde{X}_t := \begin{cases} 
X_t&\textup{ on } \{t < \tau\},\\
X_{\tau},&\textup{ on } \{t \geq \tau\} \cap G^c,\\
\liminf_{s \uparrow \tau, s \in A \cap \mathbb{Q} \times \Omega} X_s,&\textup{ on } \{t \geq \tau\} \cap G,
\end{cases}
\end{align}
where
\(
\tau := \lim_{n \to \infty} \tau_n\) and \(G := \bigcap_{n \in \mathbb{N}} \{\tau_n < \tau < \infty\}
\)
for an arbitrary announcing sequence \((\tau_n)_{n \in \mathbb{N}}\) for ~\(A\).
\begin{remark}
If \(\of 0, \rho\gs \subseteq A\) for an \(\F\)-stopping time \(\rho\), then \(\widetilde{X}^\rho = X^{\rho}\). 
\end{remark}
Clearly, in general \(\widetilde{X} \not \in \mathcal{C}\), however, if \(X \in \mathcal{M}^A_\textup{loc}(\mathbb{R}, \F, \p)\), non-negativity implies 
\(\widetilde{X} \in \mathcal{M}_\textup{loc}(\mathbb{R}, \F, \p)\), c.f. Proposition \ref{lemma tilde Z local martingale} in Appendix \ref{Extensions of non-negative local martingales on sets of interval type}.

\subsection{Characteristics of Hilbert-space-valued Semimartingales on Sets of Interval Type}
In what follows it is of little additional cost to proceed directly with Hilbert-space-valued processes. We are following closely the monographs \cite{MP80,metivier}. 
For the definition Hilbert-space-valued semimartingales, which can be done parallel to the finite-dimensional case, we refer to Appendix \ref{A HVS}.
Let \(\mathbb{H}\) be a real separable Hilbert space and
denote by \(\mathcal{N}(\mathbb{H}, \mathbb{H})\) the space of nuclear operators from \(\mathbb{H}\) into \(\mathbb{H}\).\footnote{\(\mathcal{N}(\mathbb{H}, \mathbb{H})\) is a separable Banach space, c.f. \cite{gohberg2013classes}, p.106}
Let \(A \in \mathcal{I}(\F)\).
\newpage
\begin{definition}\label{def A-c}
Let \(X\in \mathcal{S}^A(\mathbb{H}, \F, \p)\) and assume that \((B(h), C, \nu)\) is a triplet consisting of the following:
\begin{enumerate}

\item[\textup{(i)}] an \(\F\)-predictable \(\mathbb{H}\)-valued process \(B(h)\) on \(A\), 

\item[\textup{(ii)}] an \(\F\)-predictable \(\mathcal{N}(\mathbb{H}, \mathbb{H})\)-valued process \(C\) on \(A\), 
\item[\textup{(iii)}] an \(\F\)-predictable random measure \(\nu\) on \(\mathbb{R}^+ \times \mathbb{H}\) such that 
\begin{align}\label{nu dec}
\nu(\omega, \dd t, \dd x) = \1_{A \times \mathbb{H}} (\omega, t, x) \nu(\omega, \dd t, \dd x).
\end{align}

\end{enumerate}
If there exists an announcing sequence \((\tau_n)_{n \in \mathbb{N}}\) for \(A\) such that \((B(h)^{\tau_n}, C^{\tau_n}, \nu^{\tau_n})\), where 
\begin{align*}
\nu^{\tau_n} (\omega, \dd t, \dd x) := \1_{\of 0, \tau_n \gs \times \mathbb{H}}(\omega, t, x) \nu(\omega, \dd t, \dd x),
\end{align*} 
is \(\p\)-indistinguishable of the \((\F, \p)\)-characteristics of \(X^{\tau_n} \in \mathcal{S}(\mathbb{H}, \F, \p)\) for all \(n \in \mathbb{N}\)
, then we call the triplet \((B(h), C, \nu)\) the \emph{\((\F, \p, A)\)-characteristics} of \(X\). 
\end{definition}
The most important consistency observations concerning the \(A\)-characteristics, such as \emph{existence} and \emph{uniqueness}, and some additional notations, such as the \emph{continuous local martingale part on \(A\)}, are given in Appendix \ref{facts HSVSM SI}.
\subsection{Semimartingale Problems on Stochastic Invervals}\label{Semimartingale Problems on Predictable Stochastic Invervals}

We denote the Borel \(\sigma\)-field of \(\mathbb{H}\) by \(\mathcal{H}\). 
In order to deal with possibly explosive processes
 we introduce the notion of a \emph{grave}. For a Polish space \(E\), i.e. in particular for \(\mathbb{H}\) and for \(\mathcal{N}(\mathbb{H}, \mathbb{H})\),\footnote{since separable Banach spaces are Polish spaces} let \(\Delta\) be a point outside of \(E\) and denote \(E_\Delta = E\cup \{\Delta\}\).
We define \(\mathbb{D}^{E_\Delta}\) as the space of all functions \(\alpha : \mathbb{R}^+ \to E_\Delta\), such that \(\alpha(s) = \Delta,\) for all \(s \in [\xi(\alpha), \infty)\), where \(\xi(\alpha) := \inf(t \geq 0 : \alpha(t) = \Delta)\), which are \cadlag on \((0, \xi(\alpha))\).
We assume that we are given the following: 
\begin{enumerate}
\item[\textup{(i)}]
an \(\F\)-adapted 
process \(X\) with paths in \(\mathbb{D}^{\mathbb{H}_\Delta}\), 
which we call \textit{candidate process}, 
\item[\textup{(ii)}]
a triplet \((B(h), C, \nu)\), called a \textit{candidate triplet}, 
consisting of
\\[-1.3ex]
\begin{enumerate}
\item[--] an \(\mathfrak{F}\)-predictable process \(B(h)\) with paths in \(\mathbb{D}^{\mathbb{H}_\Delta}\), 
\\[-1.5ex]
\item[--] an \(\mathfrak{F}\)-predictable process \(C\) with paths in \(\mathbb{D}^{\hspace{0.04cm}\mathcal{N}(\mathbb{H}, \mathbb{H})_\Delta}\), 
\\[-1.5ex]
\item[--] an \(\mathfrak{F}\)-predictable random measure \(\nu\) on \(\mathbb{R}^+ \times \mathbb{H}\), 
\\[-1.5ex]
\end{enumerate}
\item[\textup{(iii)}]
a probability measure \(\eta\) on \((\mathbb{H}, \mathcal{H})\), which is called \textit{initial law}.
\end{enumerate}
The process \(B(h)\) may depend on the truncation function \(h\).
\begin{definition}\label{Definition Semimartingale Problem}
We call a probability measure \(\p\) on \((\Omega, \mathcal{F})\) a \emph{solution to the SMP} associated with \((\mathbb{H};\rho;\eta; X; B(h), C, \nu)\), where \(\rho\) is an \(\F\)-stopping time, \(\eta\) is an initial law, \(X\) a candidate process and \((B(h), C, \nu)\) a candidate triplet, if
\\[-1.3ex]
\begin{enumerate}
\item[\textup{(i)}] \(X^{\rho}\) is \(\p\)-indistinguishable from an \(\mathbb{H}\)-valued \((\mathfrak{F}, \p)\)-semimartingale,
\\[-1.3ex]
\item[\textup{(ii)}] the \((\F, \p)\)-characteristics of \(X^{\rho}\) 
are \(\p\)-indistinguishable from \((B^{\rho}(h), C^{\rho}, \nu^{\rho})\),
\\[-1.3ex]
\item[\textup{(iii)}] \(\p \circ X^{-1}_0 = \eta\).\\[-1.3ex]
\end{enumerate}
We call \(\p\) a solution to the SMP associated with \((\mathbb{H}; A; \eta; X; B(h), C, \nu)\), where \(A \in \mathcal{I}(\F)\), if there exists an announcing sequence \((\tau_n)_{n \in \mathbb{N}}\) for \(A\), such that for all \(n \in \mathbb{N}\), \(\p\) is a solution to the SMP associated with \((\mathbb{H};\tau_n; \eta; X; B(h), C, \nu)\).
\end{definition}
\begin{remark}
\(\p\) is a solution to the SMP associated with \((\mathbb{H}; A; X; \eta; B(h), C, \nu)\) if and only ~if 
\begin{enumerate}
\item[\textup{(i)}]
\(X|_A\) is \(\p\)-indistinguishable from a process in \(\mathcal{S}^A(\mathbb{H},\F, \p)\),
\item[\textup{(ii)}]
\((B(h), C, \nu)|_A\) is \(\p\)-indistinguishable from the \((\F, \p,A)\)-characteristics of \(X|_A\),
\item[\textup{(iii)}]
\(\p\circ X^{-1}_0 = \eta\).
\end{enumerate}
\end{remark}
\begin{proposition}\label{SMP subinterval}
\(\p\) is a solution to the SMP associated with \((\mathbb{H}; A; \eta; X; B(h), C, \nu)\) if and only if for all \(\F\)-stopping times \(\rho\) such that \(\of 0, \rho\gs\subseteq A\), \(\p\) is a solution to the SMP associated with \((\mathbb{H}; \rho; \eta; X; B(h), C, \nu)\).
\end{proposition}
\begin{proof}
The implication \(\Longleftarrow\) is trivial. The implication \(\Longrightarrow\) follows from
Corollary \ref{ordinary semimartingale} and Proposition \ref{exist A c} in Appendix \ref{facts HSVSM SI}.
\end{proof}
As an immediate consequence we obtain the following corollary.
\begin{corollary}\label{coro subinterval}
If \(\p\) is a solution to the SMP associated with \((\mathbb{H}; A; \eta; X; B(h), C, \nu)\), and \(\mathcal{I}(\F)\ni A^* \subseteq A\), then \(\p\) is also a solution to the SMP associated with \((\mathbb{H}; A^*; \eta; X;\) \(B^*(h),\) \(C^*, \nu^*)\) for any candidate triplet \((B^*(h), C^*, \nu^*)\) which coincides on \(A^*\) with \((B(h), C, \nu)\). 
\end{corollary}
Our next goal is to introduce a concept of uniqueness which is in the spirit of the \emph{local uniqueness} concept introduced by \cite{J79}, Section 12.3.e).
\begin{definition}
\begin{enumerate}
\item[\textup{(i)}]
We say that the SMP associated with \((\mathbb{H}; \rho; \eta; X; B(h), C, \nu)\) satisfies \emph{uniqueness}, if all solutions coincide on \(\mathcal{F}_{\rho-}\). 
\item[\textup{(ii)}]
Let \(\Lambda\) be a set of \(\F\)-stopping times such that for all \(\rho\in \Lambda\) we have \(\of 0, \rho\gs \subseteq A\).
We say that the SMP associated with \((\mathbb{H}; A; \eta; X; B(h), C, \nu)\) satisfies \emph{\(\Lambda\)-uniqueness} if for all \(\rho \in \Lambda\) the SMP associated with \((\mathbb{H}; \rho; \eta; X; B(h), C, \nu)\) satisfies uniqueness. 
\end{enumerate}
\end{definition}
\begin{remark}
To hope for uniqueness we usually have to assume that \((\mathcal{F}, \F)\) is \emph{generated by the candidate process \(X\)}, i.e. that
\begin{align*}
\mathcal{F} = \sigma(X_t, t \geq 0),\ \mathcal{F}^{\hspace{0.03cm}0}_t = \sigma(X_s, s\leq t)\textup{ and } \mathcal{F}_t = \bigcap_{s > t} \mathcal{F}^{\hspace{0.03cm}0}_s.
\end{align*}
The choice of the \(\sigma\)-field \(\mathcal{F}_{\rho-}\) in the definition of uniqueness has its origin in classical Markovian settings.
More precisely, standard techniques show that all solutions coincide on \(\mathcal{F}^{\hspace{0.03cm}0}_{\rho}\) instead of \(\mathcal{F}_\rho\). However, for a large class of stopping times \(\rho\) it holds that \(\mathcal{F}_{\rho-} \subseteq \mathcal{F}^{\hspace{0.03cm}0}_\rho\). 
\end{remark}
In a finite dimensional Markovian setting a related concept of uniqueness is well-studied, c.f. \cite{J79}, Section 13.3. In a non-Markovian setting we derive explicit conditions implying uniqueness in Appendix \ref{SMPs and SDEs}. 
The following remark is an immediate consequence of the definition of \emph{uniqueness}.
\begin{remark}\label{uni prop}
Let \(\Lambda\) be a set of \(\F\)-stopping times such that \(\of 0, \rho\gs \subseteq A^*\in \mathcal{I}(\F)\) for all \(\rho \in \Lambda\), let \(A\in \mathcal{I}(\F)\) such that \(A^* \subseteq A\), and let \((B(h),C,\nu)\), \((B^*(h), C^*, \nu^*)\) be two candidate triplets which coincide on \(A^*\). Then the following is equivalent:
\begin{enumerate}
\item[\textup{(i)}] The SMP associated with \((\mathbb{H}; A; \eta; X; B(h), C, \nu)\) satisfies \(\Lambda\)-uniqueness.
\item[\textup{(ii)}] The SMP associated with \((\mathbb{H}; A^*; \eta; X; B^*(h), C^*, \nu^*)\) satisfies \(\Lambda\)-uniqueness.
\end{enumerate}
\end{remark}
\section{Main Results}\label{section: Martingality of local martingales}
Let \((\Omega, \mathcal{F}, \F, \p)\) be a given filtered probability space with right-continuous filtration \(\F = (\mathcal{F}_t)_{t \geq 0}\).
We are interested in the martingality of an \(\mathbb{R}^+\)-valued local \((\F, \p)\)-martingale \(Z\) with \(\p\)-a.s. \(Z_0 = 1\).
Later we will assume that \(Z\) is driven by a Hilbert-space-valued semimartingale.
We describe the martingality of \(Z\) by convergence properties of sequences of stopping times satisfying the following convention which we impose from now on.
\begin{C}\label{conv}
Let \((\rho_n)_{n \in \mathbb{N}}\) be an increasing sequence of \(\F\)-stopping times such that \(Z^{\rho_n}\) is a uniformly integrable \((\F, \p)\)-martingale for all \(n \in \mathbb{N}\).
\end{C}
A sequence as in Convention \ref{conv} always exists, namely any localizing sequence of \(Z\).
It is important to note that the sequence in Convention \ref{conv} need not to increase \(\p\)-a.s. to infinity.
In Appendix \ref{An Integrability Condition to Indentify} we present an integrability condition to identify such sequences, c.f. \cite{CFY} for a related approach based on a Novikov-type condition.
In Section \ref{Martingality in Terms of SMPs}, we relate martingality to existence and uniqueness properties of SMPs.
In Section \ref{Martingality on Standard Systems}, under topological assumptions on the underlying filtered space, we derive sufficient conditions for the martingality of \(Z\) in terms of an almost sure convergence of ~\((\rho_n)_{n \in \mathbb{N}}\). 
\subsection{Martingality in Terms of SMPs}\label{Martingality in Terms of SMPs}
We consider five situations: Firstly, we derive a condition for the strict local martingality of \(Z\) in terms of existence properties of SMPs. Secondly, we impose the assumption of \(\Lambda\)-uniqueness for a modified SMP.
This additional structure enables us to investigate the martingality of \(Z\) by examining the limit behavior of the sequence of stopping times  \((\rho_n)_{n \in \mathbb{N}}\) under a solution to an SMP.
Thirdly, we construct an explicit sequence \((\rho_n)_{n \in \mathbb{N}}\) driven solely by \(Z\). This results in martingality conditions only depending on the path of \(Z\).
Fourthly, we consider the situation where \(Z\) is an generalized stochastic exponential. In this case the previously established integrability condition boils down to a localized Novikov-type condition, which is suitable for applications also when classical Novikov-type conditions fail to hold.
Fifthly, we present another application of our abstract results which starts with the explicit choice of \((\rho_n)_{n \in \mathbb{N}}\) to be exit times of compacta. Thanks to this assumption, we derive an explosion-type condition in the spirit of \cite{MU(2012)}.
We impose the following additional assumption. 
\begin{SA}\label{SA}
We assume that the probability measure \(\p\) is a solution to the SMP 
associated with \((\mathbb{H}; A; \eta; X; B(h), C, \nu)\). 
We denote the unique 
decomposition of \(X(h)\) 
by \(X(h) = X_0 + M(h) + B(h)\), c.f. \eqref{B(h) def} in Appendix \ref{Semimartingale Characterstics}.
\end{SA}
All following standing assumptions are only assumed to hold in the particular subsection.
\subsubsection{A Condition for Strict Local Martingality}
We now consider the situation where the underlying filtered space allows an extension of every consistent family.
This topological property is called \emph{fullness}, c.f. Definition \ref{def full} in Appendix \ref{Classical Path-Spaces}, and is for instance possessed by all standard path spaces.
It enables us to extend the consistent family \((\q_t, \mathcal{F}_t)_{t \geq 0}\) where \(\q_t := Z_t \cdot \p\) to a measure \(\q\) which is locally absolutely continuous w.r.t. \(\p\) with density process 
\(Z\). 
An application of Girsanov's theorem then yields the existence of a solution to a modified SMP. 
In view of this observation a contradiction argument yields a sufficient condition for strict local martingality. 
In the following we formalize this idea.
\begin{theorem}\label{main theorem new 2}
Assume that \((\Omega, \mathcal{F}, \F)\) is full. 
If \(Z\) is an \((\F, \p)\)-martingale, there exists a solution \(\q\) to the SMP 
associated with \((\mathbb{H}; A; \eta; X; \dot{B}(h), C, \dot{\nu})\), where on \(A\) 
\begin{equation}\label{candidtate triplet main theorem}
\begin{split}
\dot{B}(h) = B(h) + &1/Z_- \1_{\{Z_- > 0\}}\cdot \lle Z, M(h)\rre^{\p}\textit{ and } \dot{\nu} = Y \cdot \nu,\\ \textit{with } Y &= 1/Z_-\1_{\{Z_- > 0\}} M^\p_{\mu^X} (Z |\widetilde{\mathcal{P}}(\mathfrak{F})),
\end{split}
\end{equation} 
for the notation c.f. Appendix \ref{Girsanov's Theorem for Hilbert-Space-valued Semimartingales}.
Additionally we have \(\q \ll_\textup{loc} \p\) with density process \(Z\).
\end{theorem}
\begin{proof}
Thanks to the assumption that \((\Omega, \mathcal{F}, \F)\) is full, there exists a probability measure \(\q\) such that for all \(t \geq 0\) we have \(\q = \q_t = Z_t \cdot \p\) on \(\mathcal{F}_t\). 
Since for all \(t \geq 0\) it holds that \(\q_t \ll \p\) on \(\mathcal{F}_t\), we conclude that \(\q \ll_{\textup{loc}} \p\) with density process \(Z\).
Denote by \((\tau_n)_{n \in \mathbb{N}}\) an announcing sequence for \(A\). Then \(X^{\tau_n}\) is an \((\F, \p)\)-semimartingale with characteristics \((B(h)^{\tau_n}, C^{\tau_n}, \nu^{\tau_n})\), since \(\p\) solves the SMP associated with \((\mathbb{H}; A; \eta; X; B(h), C, \nu)\).
Proposition \ref{pred finie var 0} in Appendix \ref{A HVS} yields that the unique decomposition \eqref{B(h) def} of \(X(h)^{\tau_n}\) is given by \(X(h)^{\tau_n} =X_0+ M(h)^{\tau_n} + B(h)^{\tau_n}\).
Moreover, the identity \(\llbr Z, M(h)^{\tau_n}\rrbr^\p = \llbr Z, M(h)\rrbr^\p_{\cdot \wedge \tau_n}\) and again Proposition \ref{pred finie var 0} in Appendix \ref{A HVS} imply that 
\begin{align*}
\lle Z, M(h)^{\tau_n}\rre^\p = \lle Z, M(h)\rre^{\p}_{\cdot \wedge \tau_n},
\end{align*}
where \(\lle Z, M(h)^{\tau_n}\rre^\p\) is the unique \cadlag 
\(\F\)-predictable process of finite variation such that the process \(\llbr Z, M(h)^{\tau_n}\rrbr^\p - \lle Z, M(h)^{\tau_n}\rre^\p\) is a local \((\F, \p)\)-martingale, c.f. Theorem \ref{GT} in Appendix ~\ref{Girsanov's Theorem for Hilbert-Space-valued Semimartingales}.
Furthermore, Proposition II.1.30 in \cite{JS} yields that 
\begin{align*}
M^\p_{\mu^{X^{\tau_n}}}(Z|\widetilde{\mathcal{P}}(\F)) = \1_{\of 0, \tau_n\gs} M^\p_{\mu^X}(Z|\widetilde{\mathcal{P}}(\F)).
\end{align*}
Therefore, the Girsanov-type theorem given by Theorem \ref{GT} in Appendix \ref{Girsanov's Theorem for Hilbert-Space-valued Semimartingales} yields that \(X^{\tau_n}\) is also an \((\F, \q)\)-semimartingale with \((\F, \q)\)-characteristics \((\dot{B}(h)^{\tau_n}, C^{\tau_n}, \dot{\nu}^{\tau_n})\).
Now the identity \(\q \circ X_0^{-1} = \q_0 \circ X_0^{-1} = \p \circ X_0^{-1} = \eta\), as \(\p\)-a.s. \(Z_0 = 1\), yields that \(\q\) is a solution to the SMP associated with \((\mathbb{H}; \tau_n; \eta; X; \dot{B}(h), C, \dot{\nu})\).
Since this holds for all \(n \in \mathbb{N}\), the claim is proven.
\end{proof}
As a corollary we obtain the following sufficient condition for strict local ~martingality.
\begin{corollary}\label{coro strict explosion}
Let \((\Omega,\mathcal{F}, \F)\) be full. If the SMP 
associated with \((\mathbb{H}; A; \eta; X; \dot{B}(h), C, \dot{\nu})\), where \(\dot{B}(h)\) and \(\dot{\nu}\) are given by \eqref{candidtate triplet main theorem} on \(A\), has no solution, then the process \(Z\) is a strict local \((\F, \p)\)-martingale.
\end{corollary}
\begin{remark}
In the one-dimensional diffusion setting the \textit{Feller test for explosion}, c.f. \cite{Feller52} or \cite{KaraShre}, Section 5.5.C, provides a technique to apply 
Corollary \ref{coro strict explosion}.
A comprehensive discussion is given by \cite{MU(2012)}.
\end{remark}
\subsubsection{The Martingale Property under Uniqueness Assumptions}\label{main section}
In this section we impose existence and uniqueness assumptions on the modified SMP 
and derive sufficient conditions for the martingality of \(Z\).
We denote by \((\tau_n)_{n \in \mathbb{N}}\) an arbitrary announcing sequence for \(A\) and set \(\bar{A} := \bigcup_{n \in \mathbb{N}} \of 0, \tau_n \wedge \rho_n\gs \in \mathcal{I}(\F)\) and \(\tau := \lim_{n \to \infty} \tau_n\).
\begin{condition}\label{cond main uni}
We define the following conditions:
\begin{enumerate}
\item[\textup{(I)}]
The SMP associated with \((\mathbb{H}; \bar{A}; \eta; X;\dot{B}(h), C, \dot{\nu})\), where \(\dot{B}(h)\) and \(\dot{\nu}\) are given by \eqref{candidtate triplet main theorem} on \(\bar{A}\),
has a solution \(\q\) and satisfies \(\{\tau_n \wedge \rho_m, n,m \in \mathbb{N}\}\)-uniqueness.
\item[\textup{(II)}] For all \(n \in \mathbb{N}\) and \(t \geq 0\) we have \(\{\rho_n > t\}\in \mathcal{F}_{(\tau \wedge \rho_n)-}\). 
\item[\textup{(III)}] For all \(n \in \mathbb{N}\) we have \(\{\rho_n = \infty\} \in \mathcal{F}_{(\tau \wedge \rho_n)-}\).

\end{enumerate}
\end{condition}
Note that (II) and (III) are structural assumptions which relate the sequences \((\rho_n)_{n \in \mathbb{N}}\) and \((\tau_n)_{n \in \mathbb{N}}\).
Now we are in the position to give the main results of this section.
\begin{proposition}\label{main prop uni 1}
Assume \textup{(I)} and \textup{(II)} in Conditions \ref{cond main uni}. 
\begin{enumerate}
\item[\textup{(i)}]
If we have 
\begin{align}\label{Main condition}
\lim_{n \to \infty}\q(\rho_n > t) = 1,\textit{ for all } t \geq 0,
\end{align}
then \(Z\) is an \((\F, \p)\)-martingale.
\item[\textup{(ii)}]
If \(\p\)-a.s. \(\rho_n \uparrow_{n \to \infty} \infty\) and \(Z\) is an \((\F, \p)\)-martingale, then \eqref{Main condition} holds.
\end{enumerate}
\end{proposition}
\begin{proposition}\label{main prop uni 2}
Assume \textup{(I)} and \textup{(III)} in Conditions \ref{cond main uni}. 
\begin{enumerate}
\item[\textup{(i)}]
If we have 
\begin{align}\label{Main condition ui} 
\lim_{n \to \infty} \q(\rho_n = \infty) = 1,
\end{align}
then \(Z\) is a uniformly integrable \((\F, \p)\)-martingale.
\item[\textup{(ii)}]
If \(\bigcup_{n \in \mathbb{N}} \{\rho_n = \infty\}\) is a \(\p\)-a.s. event and \(Z\) is a uniformly integrable \((\F, \p)\)-martingale, then \eqref{Main condition ui} holds.
\end{enumerate}
\end{proposition}
The proofs of these two results are based on the following two lemmata.

\begin{lemma}\label{lemma stopped}
For all \(m, n \in \mathbb{N}\), the probability measure \(\q_{m, n} := Z_{\tau_m \wedge \rho_n} \cdot \p\) solves the SMP associated with \((\mathbb{H}; \tau_m \wedge \rho_n; \eta; X; \dot{B}(h), C, \dot{\nu})\), where
\(\dot{B}(h)\) and \(\dot{\nu}\) are given by \eqref{candidtate triplet main theorem} on \(\bar{A}\).
\end{lemma}
\begin{proof}
Since \(\p\)-a.s. \(Z_0 = 1\), it holds that 
\begin{align}\label{qn initial}
\q_{m,n} \circ X^{-1}_0 = \eta. 
\end{align}
Moreover, since \(X^{\tau_m \wedge \rho_n}\) is an \((\F, \p)\)-semimartingale 
and \(\q_{m,n} \ll \p\) with density process \(Z^{\tau_m \wedge \rho_n}\), the Girsanov-type theorem as given in Theorem \ref{GT} in Appendix \ref{Girsanov's Theorem for Hilbert-Space-valued Semimartingales} implies that \(X^{\tau_m \wedge \rho_n}\) is an  \((\F, \q)\)-semimartingale with \((\F, \q)\)-characteristics \((\bar{B}(h), C^{\tau_m \wedge \rho_n}, \bar{\nu})\), where 
\begin{align*}
\bar{B}(h) &= B(h)^{\tau_m \wedge \rho_n} + \1_{\{Z^{\tau_m \wedge \rho_n}_- > 0\}}/Z^{\tau_m \wedge \rho_n}_- 
\cdot \lle Z^{\tau_m \wedge \rho_n}, M(h)^{\tau_m \wedge \rho_n}\rre^\p \\
&= B(h)^{\tau_m \wedge \rho_n} + \1_{\{Z_- > 0\}}/Z_- \1_{\of 0, \tau_m \wedge \rho_n\gs} \cdot \lle Z, M(h)\rre^\p  
\\&= \dot{B}(h)^{\tau_m \wedge \rho_n},\\
\bar{\nu}(\dd t, \dd x) &= \frac{\1_{\{Z^{\tau_m \wedge \rho_n}_{t-} > 0\}}}{Z^{\tau_m \wedge \rho_n}_{t-}} M^\p_{\mu^{X^{\tau_m \wedge \rho_n}}}(Z^{\tau_m \wedge \rho_n} | \widetilde{\mathcal{P}}(\F))(t, x) \1_{\of 0, \tau_m \wedge \rho_n\gs \times \mathbb{H}}(\cdot, t, x) \nu(\dd t, \dd x)
\\&= \frac{\1_{\{Z_{t-} > 0\}}}{Z_{t-}} M^\p_{\mu^X}(Z|\widetilde{\mathcal{P}}(\F))(t,x) \1_{\of 0, \tau_m \wedge \rho_n\gs \times \mathbb{H}}(\cdot, t, x) \nu(\dd t, \dd x) 
\\&= \dot{\nu}^{\tau_m \wedge \rho_n} (\dd t, \dd x).
\end{align*}
This and \eqref{qn initial} implies that 
\(\q_{m,n}\) is a solution to the SMP 
 \((\mathbb{H};\tau_m\wedge \rho_n; \eta; X; \dot{B}(h), C, \dot{\nu})\). 
\end{proof}
\begin{lemma}\label{lemma 2}
Assume part \textup{(I)} of Conditions \ref{cond main uni}, then
\begin{align*}
\q  = Z_{\tau \wedge \rho_n} \cdot \p \textit{ on } 
\mathcal{F}_{(\tau \wedge \rho_n)-}. 
\end{align*}
\end{lemma}
\begin{proof}
Due to the assumption of \(\{\tau_m\wedge \rho_n, n, m \in \mathbb{N}\}\)-uniqueness and Lemma \ref{lemma stopped} we have 
\begin{align*}
\q =Z_{\tau_m \wedge \rho_n} \cdot \p  \textup{ on } \mathcal{F}_{(\tau_m \wedge \rho_n)-}
\end{align*}
Let \(F \in \mathcal{F}_{(\tau_m \wedge \rho_n)-}\), then 
\begin{align*}
\q(F) = \E[Z_{\tau_m\wedge \rho_n} \1_F] = \E[\E[Z_{\tau \wedge \rho_n}| \mathcal{F}_{\tau_m \wedge \rho_n}] \1_F] = \E[Z_{\tau \wedge \rho_n} \1_F],
\end{align*}
due to Doob's stopping theorem. 
This proves that \(\q = Z_{\tau \wedge \rho_n} \cdot \p\) on \(\bigcup_{m \in \mathbb{N}} \mathcal{F}_{(\tau_m \wedge \rho_n)-}\). Then a monotone class argument 
yields 
\begin{align*}
\q = Z_{\tau \wedge \rho_n} \cdot \p\ \textup{ on } \bigvee_{m \in \mathbb{N}} \mathcal{F}_{(\tau_m \wedge \rho_n)-}. 
\end{align*}
Finally, thanks to \cite{DellacherieMeyer78}, Theorem IV.56, \(\bigvee_{m \in \mathbb{N}} \mathcal{F}_{(\tau_m \wedge \rho_n)-} = \mathcal{F}_{(\tau \wedge \rho_n)-}\),
which finishes the proof.
\end{proof}
\hspace{-0.45cm}\textit{Proof of Proposition \ref{main prop uni 1}:}
(i). 
Due to part (II) of Conditions \ref{cond main uni}, Lemma \ref{lemma 2} and Doob's stopping theorem we obtain 
\begin{align}\label{q comp}
\q(\rho_n > t) &= \E\big[Z_{\tau \wedge \rho_n} \1_{\{\rho_n > t\}}\big] = \E\big[ \E\big[Z_{\rho_n}|\mathcal{F}_{\tau \wedge \rho_n}\big]\1_{\{\rho_n > t\}}\big] = \E\big[Z_{\rho_n} \1_{\{\rho_n > t\}}\big].
\end{align}
Therefore, Lemma \ref{neues lemma} (i) yields the claim.
(ii). Since \(\p\)-a.s. \(\rho_n \uparrow_{n \to \infty} \infty\), Lemma \ref{ruf mimic} and \eqref{q comp} yield the claim.
\qed
\\\\
\hspace{-0.45cm}\textit{Proof of Proposition \ref{main prop uni 2}:}
(i). 
The claim follows identically to the proof of Proposition \ref{main prop uni 1} ~(i).

(ii). 
Part (III) of Conditions \ref{cond main uni}, Lemma \ref{lemma 2}, Doob's optional stopping theorem, dominated convergence and the assumption that \(\bigcup_{n \in \mathbb{N}} \{\rho_n = \infty\}\) is a \(\p\)-a.s. event yield 
\begin{align*}
\lim_{n \to \infty} \q(\rho_n = \infty) = \lim_{n \to \infty} \E[Z_{\rho_n} \1_{\{\rho_n = \infty\}}] =
 \lim_{n \to \infty} \E[Z_\infty \1_{\{\rho_n = \infty\}}] = \E[Z_\infty] = 1.
\end{align*}
This finishes the proof.
\qed
\begin{remark}
For explicit conditions implying existence of a solution to the modified SMP we refer to Theorem \ref{theorem extension2} below.
\end{remark}

\subsubsection{A first explicit definition of \((\rho_n)_{n \in \mathbb{N}}\)}\label{MSE}
In this section we benefit from a boundedness condition implying uniformly integrable martingality due to \cite{J79} in order to construct an explicit sequence \((\rho_n)_{n \in \mathbb{N}}\).

Let \((\sigma_n)_{n \in \mathbb{N}}\) be a sequence of \(\F\)-stopping times, denote \(A' = \bigcup_{n \in \mathbb{N}} \of 0, \sigma_n\gs\), and let \(Z \in \mathcal{M}^{A'}_\textup{loc}(\mathbb{R}, \F, \p)\) be non-negative.
We extend \(Z\) as in \eqref{extension} and denote the extension by \(\widetilde{Z}\). Due to Proposition \ref{lemma tilde Z local martingale} in Appendix \ref{Extensions of non-negative local martingales on sets of interval type}, \(\widetilde{Z}\) is a non-negative local \((\F, \p)\)-martingale.
We set 
\begin{align}\label{C(Z)}
C(\widetilde{Z}) := \lle \widetilde{Z}^c, \widetilde{Z}^c\rre^{\p} + \sum_{s \leq \cdot} \bigg(\widetilde{Z}_{s-}- \sqrt{\widetilde{Z}_s \widetilde{Z}_{s-}}\hspace{0.05cm}\bigg)^2.
\end{align}
Thanks to Lemma 8.5 in \cite{J79} it holds that \(C(\widetilde{Z}) \in \mathcal{A}^+_\textup{loc}(\F, \p)\), implying that \(C(\widetilde{Z})^p\), the \((\F, \p)\)-compensator of \(C(\widetilde{Z})\), is well-defined.
Moreover, we define the sequence \((\rho_n)_{n \in \mathbb{N}}\) of \(\F\)-stopping times by
\begin{align}\label{gamma}
\rho_n &:= \inf\big(t \geq 0 : \1_{\{\widetilde{Z}_- > 0\}}/\widetilde{Z}_-^2 \cdot C(\widetilde{Z})^p_{t} \geq n\big). 
\end{align}
\begin{condition}\label{gen bound cond}
We define the following conditions:
\begin{enumerate}
\item[\textup{(I)}] We have \(A' \subseteq A\).
\item[\textup{(II)}] The random time \(\sigma := \lim_{n \to \infty} \sigma_n\) is an \(\F\)-predictable time.
\item[\textup{(III)}]
For each \(n \in \mathbb{N}\) there exists a constant \(c_n > 0\) such that 
\begin{align}\label{jump condition}
\frac{\Delta C(\widetilde{Z})^p_{\rho_n}}{\widetilde{Z}_{\rho_n-}^2} \1_{\{\widetilde{Z}_{\rho_n-} > 0\}} \leq c_n. 
\end{align}
\end{enumerate}
\end{condition}
The main result of this section is the following:
\begin{proposition}\label{gen coro extended stoch exp}
Assume that the 
Part \textup{(I)} of Conditions \ref{cond main uni} and \ref{gen bound cond} hold. 
\begin{enumerate}
\item[\textup{(i)}]
If for all \(t \geq 0\) we have 
\begin{align}\label{Q cont assup}
\q\left(\ \frac{\1_{\{\widetilde{Z}_- > 0\}}}{\widetilde{Z}_-^2} \cdot C(\widetilde{Z})^p_t < \infty\right) = 1,
\end{align}
then \(\widetilde{Z}\) is an \((\F, \p)\)-martingale.
\item[\textup{(ii)}]
If we have 
\begin{align*}
\q\left(\ \frac{\1_{\{\widetilde{Z}_- > 0\}}}{\widetilde{Z}_-^2} \cdot C(\widetilde{Z})^p_\infty< \infty\right) = 1, 
\end{align*}
then  \(\widetilde{Z}\) is a uniformly integrable \((\F, \p)\)-martingale.
\end{enumerate}
\end{proposition}
Before the give a proof we shortly discuss an example.
\\\\
\textbf{Example.}
We now consider a finite-dimensional continuous and Markovian setting to illustrate an application of Proposition \ref{gen coro extended stoch exp}.
Let \(\beta= (\beta^i)_{i \leq d}: \mathbb{R}^+ \times \mathbb{R}^d \to \mathbb{R}^d\) and \(b = (b^i)_{i \leq d} : \mathbb{R}^+ \times \mathbb{R}^d \to \mathbb{R}^d\) be Borel functions which are locally bounded, i.e. for each \(n \in \mathbb{N}\) bounded on \([0, n] \times \{x \in \mathbb{R}^d : |x| \leq n\}\), and \(a=(a^{ij})_{i, j \leq d} : \mathbb{R}^+ \times \mathbb{R}^d \to \mathbb{S}^d_+\) be continuous, where \(\mathbb{S}^d_+\) denotes the space of positive definit symmetric \(d \times d\) matrices.
Next we specify the SMP of Standing Assumption \ref{SA}.
We assume that \(A = \of 0, \infty\of\), that the underlying filtered space \((\Omega, \mathcal{F}, \F)\) is the path space \((\mathbb{C}^d, \mathcal{C}^d, \mathfrak{C}^d)\) and that \(X\) is the coordinate process \(X_t(\omega) = \omega\hspace{0.01cm}_t\). Moreover, for \(\omega \in \mathbb{C}^d\) we set
\begin{align*}
B (\omega) = b(\omega)\cdot I,\quad C (\omega) = a(\omega)\cdot I, \quad \nu = 0,
\end{align*}
where \(b(\omega)_t := b(t, \omega\hspace{0.01cm}_t)\) and \(a(\omega)_t := a(t, \omega\hspace{0.01cm}_t)\).
Due to the local boundedness assumptions on \(a\) and \(b\), the SMP associated with \((\mathbb{R}^d; A; \varepsilon_{x_0};X; B, C, 0)\) is equivalent to the \emph{classical martingale problem} of \cite{SV} associated to the operator
\begin{align*}
L_t f (x) = \sum_{i \leq d} b^i(t, x) \partial_i f (x)+ \frac{1}{2} \sum_{i, j \leq d} a^{ij}(t, x) \partial^2_{ij} f (x),
\end{align*}
and the initial law \(\varepsilon_{x_0}\) for \(x_0 \in \mathbb{R}^d\),
c.f. Theorem 13.55 in \cite{J79} and Proposition 5.4.11 in \cite{KaraShre}.
Sufficient conditions for the existence of a solution can for instance be found in \cite{IW89,J79,SV}.
Now we are interested in the martingality of the local \((\mathfrak{C}^d, \p)\)-martingale
\begin{align*}
Z = \mathcal{E}\left(\beta(X) \cdot X^c\right),
\end{align*}
where \(\beta(X)_t := \beta(t, X_{t})\). The local martingality follows immediately from the local boundedness assumption of \(\beta\).
\begin{corollary}\label{coro markov}
\(Z\) is a \((\mathfrak{C}^d, \p)\)-martingale if and only if the SMP \((\mathbb{R}^d; A; \varepsilon_{x_0};X; \dot{B}, \dot{C}, 0)\), where
\begin{align}\label{triplet example}
\dot{B}(\omega) = b (\omega) \cdot I + (\beta a)(\omega)  \cdot I \textit{ and } (\beta a)(\omega)_t := \beta(t,\omega\hspace{0.01cm}_t) a(t,\omega\hspace{0.01cm}_t),
\end{align} 
has a solution.
\end{corollary}
\begin{proof}
Since \((\mathbb{C}^d, \mathcal{C}^d, \mathbf{C}^d)\) is full, c.f. Remark \ref{bichteler remark full} in Appendix \ref{Classical Path-Spaces}, the implication \(\Longrightarrow\) follows from Theorem \ref{main theorem new 2}.
We now show the converse implication. 
Note that the Conditions ~\ref{gen bound cond} are trivially satisfied and that 
\begin{align*}
\frac{\1_{\{Z_- > 0\}}}{Z^2_-} \cdot C(Z)^p_t = (\beta c \beta^*)(X) \cdot I_t,\ \textup{ where } \ (\beta c \beta^*)(X) := \sum_{i, j \leq d} \beta^i(X) c^{ij}(X) \beta^j(X),
\end{align*}
which is clearly finite for each \(t \geq 0\) if \(X\) is a continuous process. Hence \eqref{Q cont assup} is satisfied
and the claim follows from Proposition \ref{gen coro extended stoch exp} if the SMP associated with \((\mathbb{R}^d; A; \varepsilon_{x_0}; \dot{B}, C, 0)\) satisfied \(\Upsilon := \{n \wedge \rho_m, n, m \in \mathbb{N}\}\)-uniqueness,
where \(\rho_m = \inf(t \geq 0 : (\beta c \beta^*)(X) \cdot I_t \geq m)\).
Due to Exercise 13.12 in \cite{J79} the SMP associated with \((\mathbb{R}^d; A; \varepsilon_{x_0};X; \dot{B}, C, 0)\) satisfies \(\mathfrak{P}\)-uniqueness, where \(\mathfrak{P}\) denotes the set of \(\mathfrak{C}^d\)-predictable times.
Thanks to Proposition I.2.13 in \cite{JS} it holds that \(\Upsilon \subseteq \mathfrak{P}\) and the proof is finished.
\end{proof}
In a one-dimensional setting this corollary is contained in Corollary 2.2 in \cite{MU(2012)}.
We shortly give a simple example where Corollary \ref{coro markov} implies martingality and Novikov's and Kazamaki's conditions fail.
Let \(d = 1\) and set \(x_0 = 0, \beta(x) = x, b = 0\) and \(a = 1\), i.e. let \(X\) be a one-dimensional standard Brownian motion and \(Z = \mathcal{E}(X\cdot X)\). 
Then the SMP associated to \((\mathbb{R}; A; \varepsilon_0; X;\beta (X)\cdot I, I, 0)\) has a solution due to classical Lipschitz conditions.
Hence \(Z\) is a \((\mathfrak{C}^d, \p)\)-martingale. However, for large enough \(t \geq 0\) it holds that \(\E[\exp(X/2 \cdot X_t)] = \infty\) which implies that Novikov's and Kazamaki's conditions fail.

The invertibility assumptions on \(a\) may be exchanged by a global uniqueness assumption. More precisely, it follows from Theorem 12.73 in \cite{J79} and Exercise 6.7.4 in \cite{SV} that well-posedness, i.e. the existence of a unique global solution, and locally bounded coefficients imply \(\mathfrak{P}\)-uniqueness.
We state this observation in the following corollary.
\begin{corollary}
Assume that \(\beta\) and \(b\) are as above and \(a : \mathbb{R}^+ \times \mathbb{R}^d \to \mathbb{S}^d\) is a locally bounded Borel function, where \(\mathbb{S}^d\) denotes the set of non-negative definite symmetric \(d \times d\)-matrices. If the SMP associated with \((\mathbb{R}^d; A; \varepsilon_{x_0}; X;\dot{B}, C, 0)\), where \(\dot{B}\) is given as in \eqref{triplet example}, is well-posed, then \(Z\) is a \((\mathfrak{C}^d, \p)\)-martingale. 
\end{corollary}
For an extension beyond Markovianity and finite dimensions we refer to Section \ref{A Case Study - Martingality in Terms of SDEs driven by Hilbert-Space-Valued Brownian Motion} below.
We now turn to the proof of Proposition \ref{gen coro extended stoch exp} which is based on the following lemma.
\begin{lemma}\label{con extended loc mart}
If Conditions \ref{gen bound cond} \textup{(II), (III)} hold, then for all \(n \in \mathbb{N}\) the process \(\widetilde{Z}^{\rho_n}\) is a uniformly integrable \((\F, \p)\)-martingale.
Moreover, we have for all \(t \geq 0\),
\begin{center}
\(
\{\rho_n > t\} \in \mathcal{F}_{(\sigma \wedge \rho_n)-} 
\)
and \(\{\rho_n = \infty\} \in \mathcal{F}_{(\sigma \wedge \rho_n)-}\).
\end{center}
\end{lemma}
\begin{proof}
Note that 
\begin{align*}
\frac{\1_{\{\widetilde{Z}_- > 0\}}}{\widetilde{Z}_-^2} \cdot C(\widetilde{Z})^p_{\rho_n} \leq n + c_n,
\end{align*}
c.f. part (III) of Conditions \ref{gen bound cond}. Hence it follows from Lemma \ref{besser als novi lemma} in Appendix \ref{An Integrability Condition to Indentify} that \(\widetilde{Z}^{\rho_n}\) is a uniformly integrable \((\F, \p)\)-martingale. 
Due to construction we have \(\widetilde{Z} = \widetilde{Z}^\sigma\). 
This yields 
\begin{align*}
\frac{\1_{\{\widetilde{Z}_- > 0\}}}{\widetilde{Z}_-^2} \cdot C(\widetilde{Z})^p = \frac{\1_{\{\widetilde{Z}_- > 0\}}}{\widetilde{Z}_-^2} \cdot C (\widetilde{Z})^p_{\cdot \wedge \sigma} =: U.
\end{align*}
Since \(U\) is \(\F\)-predictable, Proposition I.2.4 in \cite{JS} yields that for all \(t \geq 0\), the random variable \(U_t\) is \(\mathcal{F}_{\sigma-}\)-measurable.
Each \(\rho_n\) is \(\F\)-predictable thanks to part (III) of Conditions \ref{gen bound cond} and Proposition I.2.13 in \cite{JS}. Therefore, since \(\sigma\) is also \(\F\)-predictable, c.f. part (II) of Conditions \ref{gen bound cond}, we have 
\begin{align}\label{pred cap F}
\mathcal{F}_{(\sigma \wedge \rho_n)-} = \mathcal{F}_{\sigma-} \cap \mathcal{F}_{\rho_n-},\end{align}
c.f. \cite{DellacherieMeyer78}, p. 119.
Therefore, since \(U\) is an increasing process, we obtain \(\{\rho_n > t\} = \{U_t < n\} \in \mathcal{F}_{\sigma -}\). 
This, together with \eqref{pred cap F} and \(\{\rho_n > t\} \in \mathcal{F}_{\rho_n -}\), yields that \(\{\rho_n > t\} \in \mathcal{F}_{(\sigma \wedge \rho_n)-}\).

Since 
\(U_{\sigma} \1_{\{\sigma < \infty\}}\) is \(\mathcal{F}_{\sigma-}\)-measurable, c.f. Proposition I.2.4 in \cite{JS}, and \(\{\sigma < \infty\} \in \mathcal{F}_{\sigma -}\), we also obtain that
\begin{align*}
\{\sigma< \infty\} &\cap \{\rho_n = \infty\} 
= \{\sigma < \infty\}\cap \{U_\sigma \1_{\{\sigma < \infty\}} < n\} \in \mathcal{F}_{\sigma -}.
\end{align*}
This, together with 
\begin{align*} 
\{\sigma = \infty\} \cap \{\rho_n = \infty\} = \{\sigma \wedge \rho_n = \infty\} \in \mathcal{F}_{(\sigma \wedge \rho_n) -} \subseteq \mathcal{F}_{\sigma-}
\end{align*}
yields that \(\{\rho_n = \infty\} \in \mathcal{F}_{\sigma-}\). Employing again \eqref{pred cap F} and \(\{\rho_n = \infty\} \in \mathcal{F}_{\rho_n -}\), we conclude that \(\{\rho_n = \infty\} \in \mathcal{F}_{(\sigma \wedge \rho_n)-}\).
This finishes the proof.
\end{proof}

\hspace{-0.6cm}
\textit{Proof of Proposition \ref{gen coro extended stoch exp}:}
Part (I) of Conditions \ref{gen bound cond} implies \(\sigma \leq \lim_{n \to \infty} \tau_n = \tau\).
Therefore \(\mathcal{F}_{(\sigma\wedge \rho_n)-} \subseteq \mathcal{F}_{(\tau \wedge \rho_n)-}\). 
Note that 
\begin{align*}
\{\rho_n > t\} &= \big\{\1_{\{\widetilde{Z}_- > 0\}}/\widetilde{Z}_-^2 \cdot C(\widetilde{Z})^p_t < n\big\},\\ 
 \{\rho_n = \infty\} &= \big\{\1_{\{\widetilde{Z}_- > 0\}}/\widetilde{Z}_-^2 \cdot C(\widetilde{Z})^p_\infty < n,\ \forall t \geq 0\big\}.
\end{align*}
Now the claim of (i) follows from 
Lemma \ref{con extended loc mart} and Proposition \ref{main prop uni 1}, and the claim of (ii) follows from Lemma \ref{con extended loc mart},
Proposition \ref{main prop uni 2}, and the identity
\begin{align*}
\bigcup_{n \in \mathbb{N}} \{\rho_n = \infty\} &= \big\{\1_{\{\widetilde{Z}_- > 0\}}/\widetilde{Z}_-^2 \cdot C(\widetilde{Z})^p_\infty < \infty \big\}.
\end{align*}
This finishes the proof. \qed

\subsubsection{Martingality of Generalized Stochastic Exponentials}\label{Martingality of Generalized Stochastic Exponentials}
Let us now expand the setting considered by \cite{KLS-LACOM2} in their study of local absolute continuity of laws of semimartingales by allowing for stochastic lifetimes.
The continuous case was studied by \cite{RufSDE}. 
Here we derive condition for a general semimartingale setting.
We assume 
the following:
\begin{enumerate}
\item[\textup{(I)}]
\(\mathbb{H} = \mathbb{R}^d\) for \(d \in \mathbb{N}\).
\item[\textup{(II)}] 
For any announcing sequence \((\tau_n)_{n \in \mathbb{N}}\) for \(A\), the random time \(\tau := \lim_{n \to \infty}\tau_n\) is an \(\F\)-predictable time.
\item[\textup{(III)}]
We choose a \emph{good version} of \(C\), i.e. we assume that \(C = c \cdot \bar{c}\), where \(c\) is \(\F\)-predictable and \(\bar{c}\) is real-valued, continuous and increasing.
\item[\textup{(IV)}]
We have identically \(a_t := \nu(\{t\} \times \mathbb{R}^d) \leq 1\) for all \(t \geq 0\).
\item[\textup{(V)}]
We are given a positive \(\mathcal{P}(\F) \otimes \mathcal{B}^d\)-measurable function \(U\) such that 
\begin{align*}
\widehat{U}_t := \int_{\mathbb{R}^d} U(t, x)\nu(\{t\} \times \dd x) \leq 1,
\end{align*}
for all \(t \geq 0\), and \(\{a = 1\} \subseteq \{\widehat{U} = 1\}\).
\item[\textup{(VI)}]
We are given an \(\F\)-predictable \(\mathbb{R}^d\)-valued process \(K\). The transposed of \(K\) is denoted by \(K^*\). 
\end{enumerate}
We define by \(R\) a process on \(A\) such that 
\begin{align}\label{R}
&R^{\tau_n} =KcK^* \cdot \bar{c}_{\cdot \wedge \tau_n} + \big(1 - \sqrt{U}\hspace{0.02cm}\big)^2* \nu_{\cdot \wedge \tau_n} + \sum_{s \leq \cdot \wedge \tau_n} \bigg( \sqrt{1 - a_s} - \sqrt{1 - \widehat{U}}\ \bigg)^2,
\end{align}
and denote its extension \eqref{extension} by \(\widetilde{R}\).
We furthermore set
\begin{align*}
U' := U - 1 + \frac{\widehat{U} - a}{1 - a}\textup{ and } \rho_n := \inf(t \geq 0 : \widetilde{R}_{t} \geq n).
\end{align*}
\begin{SA}\label{SA2}
Assume that \(\widetilde{R}\) \emph{does not jump to infinity}, i.e. \(\widetilde{R}_{\rho-} = \infty\) on \(\{\rho < \infty\}\), where \(\rho := \inf(t \geq 0 : \widetilde{R}_{t} = \infty)\).
\end{SA}
\begin{remark}\label{pred remark}
Standing Assumption \ref{SA2} implies that \(\widetilde{R}\) is right-continuous and hence that \((\rho_n)_{n \in \mathbb{N}}\) is a sequence of \(\F\)-stopping times.
Moreover, for all \(n \in \mathbb{N}\) we have \(\rho_n < \rho\) on \(\{\rho < \infty\}\). Therefore \(\rho\) is an \(\F\)-predictable time as it can be fortelled by the sequence \((\rho_n \wedge n)_{n \in \mathbb{N}}\) and \(\{\rho = 0\} = \emptyset\), c.f. \cite{DellacherieMeyer78}, Theorem IV.71.
\end{remark}
Next we 
define the process
\begin{align*}
N^{\tau_n \wedge \rho_n} := K \1_{\of 0, \tau_n \wedge  \rho_n\gs} \cdot X^c + U' \1_{\of 0, \tau_n \wedge \rho_n\gs\times \mathbb{R}^d} * (\mu^X - \nu),
\end{align*}
whose well-definedness is established in the following lemma.
\begin{lemma}\label{Lemma N}
For all \(n \in \mathbb{N}\) the process \(N^{\tau_n \wedge \rho_n}\) is a local \((\F, \p)\)-martingale. Moreover, it holds that \(\Delta N^{\tau_n \wedge \rho_n} > -1\).
\end{lemma}
\begin{proof}
Since \(\rho_n < \rho\) on \(\{\rho < \infty\}\) we obtain 
\begin{align}\label{Delta R bound}
\Delta \widetilde{R}_{\tau_n \wedge \rho_n} = 2\bigg(1 - \widehat{\sqrt{U_{\tau_n \wedge \rho_n}}} - \sqrt{(1 - a_{\tau_n \wedge \rho_n})(1 - \widehat{U}_{\tau_n \wedge \rho_n})}\ \bigg) \leq 2,
\end{align}
where we denote
\begin{align*}
\widehat{\sqrt{U_{\tau_n \wedge \rho_n}}} := \int_{\mathbb{R}^d} \sqrt{U(\tau_n \wedge \rho_n, x)} \nu(\{\tau_n \wedge \rho_n\} \times \dd x).
\end{align*}
Hence it holds that
\begin{align}\label{R bound}
\widetilde{R}_{\tau_n \wedge \rho_n} \leq n + 2.
\end{align}
This yields 
 \(K\1_{ \of 0, \tau_n \wedge \rho_n\gs} \in L^2(X_{\cdot \wedge \tau_n}^c, \F, \p)\).
Note that for \((\omega, t) \in \of 0, \tau_n \wedge \rho_n\gs\) we have
\begin{equation}\label{delta N}
\begin{split}
\widetilde{U}'_t (\omega):\hspace{-0.13cm}&= U'(\omega, t, \Delta X_t(\omega)) \1_{\{\Delta X_t(\omega) \not= 0 \}} - \int_{\mathbb{R}^d} U'(\omega, t, x) \nu(\omega, \{t\} \times \dd x)
\\&= (U(\omega, t, \Delta X_t(\omega)) - 1) \1_{\{\Delta X_t(\omega) \not = 0\}} - \frac{\widehat{U}_t (\omega)- a_t(\omega)}{1 - a_t(\omega)}\1_{\{\Delta X_t(\omega) = 0\}},
\end{split}
\end{equation}
with the convention that \(0/0 = 0\).
Since \(U\) is positive, we have \(\widetilde{U}' > -1\) on \(\of 0, \tau_n \wedge \rho_n\gs \cap \{\Delta X \not = 0\}\).
The assumption that identically \(\widehat{U} \leq 1\) implies \(\widehat{U} - a \leq 1 - a\), where equality holds on \(\{\widehat{U} = 1\}\). Thanks to the assumption that \(\{a = 1\}\subseteq \{\widehat{U} = 1\}\) and the convention that \(0/0 = 0\), 
we obtain that \(\widetilde{U}' > -1\) on \(\of 0, \tau_n \wedge \rho_n\gs\cap\{\Delta X= 0\}\), from which
we deduce that \(\widetilde{U}' > -1\) on \(\of 0, \tau_n \wedge \rho_n\gs\). 
Due to Equation 12.41 in \cite{J79} we have on \(\of 0, \tau_n \wedge \rho_n\gs\)
\begin{equation}\label{12.41}
\begin{split}
\bigg(1 - \sqrt{1 + U' - \widehat{U}'}\ &\bigg)^2 * \nu + \sum_{s \leq \cdot}(1 - a_s)\bigg(1 - \sqrt{1 - \widehat{U}'}\ \bigg)^2 
\\&= \big(1 - \sqrt{U}\big)^2 * \nu + \sum_{s \leq \cdot} \bigg(\sqrt{1 - a_s} - \sqrt{1 - \widehat{U}}\ \bigg)^2.
\end{split}
\end{equation}
Thanks to Theorem II.1.33 in \cite{JS} and the bound \eqref{R bound} we conclude that \(U' \1_{\of 0,\tau_n \wedge \rho_n\gs\times \mathbb{R}^d} \in G_\textup{loc}(\mu^{X^{\tau_n}}, \F, \p).\)
This implies that \(N^{\tau_n \wedge \rho_n}\) is a well-defined local \((\F, \p)\)-martingale.
Due to the definition of the integral process \(U' \1_{\of 0, \tau_n \wedge \rho_n\gs \times \mathbb{R}^d} * (\mu^X - \nu)\), we have \(\Delta N^{\tau_n \wedge \rho_n} = \widetilde{U}'_{\cdot \wedge \tau_n \wedge \rho_n} >-1\).
This finishes the proof.
\end{proof}
In view of this lemma we can define a non-negative local \((\F, \p)\)-martingale on \(\bar{A} := \bigcup_{n \in \mathbb{N}} \of 0, \tau_n \wedge \rho_n\gs\) by \(Z := \mathcal{E}(N)\). 
Now we extend \(Z\) to \(\widetilde{Z}\) in the same manner as in \eqref{extension}. The extension \(\widetilde{Z}\) is a non-negative local \((\F, \p)\)-martingale, c.f. Proposition \ref{lemma tilde Z local martingale} in Appendix ~\ref{Extensions of non-negative local martingales on sets of interval type}.
\begin{remark}\label{pred remark}
In the notation of the previous section we have \(\sigma_n = \tau_n \wedge \rho_n\).
Since \(\tau = \lim_{n \to \infty} \tau_n\) is assumed to be \(\F\)-predictable, and \(\rho\) is \(\F\)-predictable due to Standing Assumption \ref{SA2}, c.f. Remark \ref{pred remark}, the random time \(\lim_{n \to \infty} \sigma_n  
= \tau \wedge \rho\) is \(\F\)-predictable.
\end{remark}
The following corollary generalizes one direction of Theorem 3.3 in \cite{RufSDE}.
For another generalization we refer to Section \ref{Martingality on Standard Systems} below.
\begin{corollary}\label{coro j 1}
Assume that the SMP associated with \((\mathbb{R}^d; \bar{A}; \eta; X; \dot{B}(h), C, \dot{\nu})\), where on \(\bar{A}\)
\begin{equation}\label{applied triplet}
\begin{split}
\dot{B}(h) &= B(h) + KcK^* \cdot \bar{c} + h(x)(U - 1) * \nu\textit{ and } \dot{\nu} =  U\cdot \nu,
\end{split}
\end{equation}
has a solution \(\q\) and satisfies \(\{\tau_n \wedge \rho_m, n, m \in \mathbb{N}\}\)-uniqueness. 
\begin{enumerate}
\item[\textup{(i)}] If for all \(t \geq 0\) we have 
\begin{align*}
\q\big(\widetilde{R}_{t \wedge \rho} < \infty\big) = 1
\end{align*}
then \(\widetilde{Z}\) is an \((\F, \p)\)-martingale.
\item[\textup{(ii)}]
If we have 
\begin{align*}
\q\big(\widetilde{R}_{\rho} < \infty\big) = 1
\end{align*}
then \(\widetilde{Z}\) is a uniformly integrable \((\F, \p)\)-martingale.
\end{enumerate}
\end{corollary}
The proof of this corollary is based on the following lemma and Proposition \ref{gen coro extended stoch exp}.
\begin{lemma}\label{iden R lemma}
We have 
\begin{align}\label{R = C stuff}
\widetilde{R}^{\rho} = \frac{\1_{\{\widetilde{Z}_- > 0\}}}{\widetilde{Z}_-^2} \cdot C(\widetilde{Z})^p.
\end{align}
\end{lemma}
\begin{proof}
For each \(n \in \mathbb{N}\) we obtain that on \(\of 0, \tau_n \wedge \rho_n\gs\)
\begin{align*} 
\1_{\{\widetilde{Z}_- > 0\}}/\widetilde{Z}_-^2 \cdot C(\widetilde{Z})^p 
&= \1_{\{\mathcal{E}(N^{\tau_n \wedge \rho_n})_- > 0\}}/\mathcal{E}(N^{\tau_n \wedge \rho_n})^2_- \cdot C(\mathcal{E}(N^{\tau_n \wedge \rho_n}))^p
= \widetilde{R},
\end{align*}
where we used that identically \(\mathcal{E}(N^{\tau_n \wedge \rho_n})_- > 0\), c.f. Theorem I.4.61 in \cite{JS}, Lemma 8.8 in \cite{J79}, and \eqref{12.41}.
Therefore the identity \eqref{R = C stuff} holds on \(\bar{A}\).
Due to construction, we have 
\begin{align*}
\widetilde{R}^\rho = \widetilde{R}^{\tau \wedge \rho}\textup{ and } \1_{\{\widetilde{Z}_- > 0\}}/\widetilde{Z}_-^2 \cdot C(\widetilde{Z})^p = \1_{\{\widetilde{Z}_- > 0\}}/\widetilde{Z}_-^2 \cdot C(\widetilde{Z})^p_{\cdot \wedge \tau \wedge \rho}.
\end{align*}
Therefore it remains to study \eqref{R = C stuff} on 
\begin{equation}
\label{remains}
\begin{split}
\big(\bar{A} &\ \cup\ \gs \tau \wedge \rho, \infty\of\big)^c 
= \begin{cases}
\of \tau\gs& \textup{ on } \{\tau < \rho\} \cap \big(\bigcap_{n \in \mathbb{N}} \{\tau_n < \tau < \infty\}\big) \times \mathbb{R}^+,
\\
\of \rho \gs&\textup{ on }
\{\rho \leq \tau\} \cap  \big(\{ \rho < \infty\}
 \times \mathbb{R}^+\big),
\\
\ \emptyset&\textup{ otherwise}.
\end{cases}
\end{split}
\end{equation}
In the first two cases the definition of the extension \eqref{extension} and the \cadlag paths of \(\widetilde{Z}\) imply \(\widetilde{Z}_{\tau \wedge \rho} = \widetilde{Z}_{(\tau \wedge \rho)-}\). Hence
\(\Delta C(\widetilde{Z})^p_{\tau \wedge \rho} = 0\), using I.3.21 and Theorem I.2.28 (ii) in \cite{JS}. 
In the first case of \eqref{remains}, this implies that the identity \eqref{R = C stuff} holds due to the definition of \(\widetilde{R}\), which yields that \(\Delta \widetilde{R}_{\tau \wedge \rho} = 0\). 
In the second case of \eqref{remains},
Standing Assumption \ref{SA2} yields
that \(\Delta \widetilde{R}_{\tau \wedge \rho} = 0\). 
This finishes the proof. 
\end{proof}
\hspace{-0.57cm}
\textit{Proof of Corollary \ref{coro j 1}:}
We set \(\sigma_n = \tau_n \wedge \rho_n\), then part (I) of Conditions \ref{gen bound cond} holds obviously, and part (II) of Conditions \ref{gen bound cond} holds due to Remark \ref{pred remark}.
It follows by the same arguments as in the proof of Theorem III.3.24 and Lemma III.5.27 in \cite{JS}, that on \(\bar{A}\) we have 
\begin{equation}\label{ausgerechnet}
\begin{split}
1/ \widetilde{Z}_- \1_{\{\widetilde{Z}_- > 0\}} M^{\p}_{\mu^X}(\widetilde{Z} |\widetilde{\mathcal{P}}(\F)) &= U,\\
1/\widetilde{Z}_-\1_{\{\widetilde{Z}_- > 0\}} \cdot \lle \widetilde{Z}, M(h)\rre^\p &=KcK^* \cdot \bar{c}+ h(x) (U- 1) * \nu. 
\end{split}
\end{equation}
Moreover, it follows similarly to \eqref{Delta R bound}, that for all \(n, m \in \mathbb{N}\), \(\Delta \widetilde{R}_{\tau_m \wedge \rho_n} \leq 2\). Therefore, 
\begin{align*}
\Delta \widetilde{R}_{\rho_n}   
&= \sum_{m \in \mathbb{N}} \Delta \widetilde{R}_{\tau_m \wedge \rho_n} \1_{\gs \tau_{m-1}, \tau_m\gs}(\cdot, \rho_n) \leq 2,
\end{align*}
where \(\tau_0 := 0\). Hence, in view of Lemma \ref{iden R lemma} and noting that 
\begin{align*}
\rho_n = \inf(t \geq 0 : \widetilde{R}^{\rho} \geq n),
\end{align*}
we conclude that part (III) of Conditions \ref{gen bound cond} holds.
Now the claim follows from the identities \eqref{ausgerechnet}, Lemma \ref{iden R lemma} and Proposition ~\ref{gen coro extended stoch exp}.
\qed

\subsubsection{An Explosion Condition}\label{Explosion criteria}

In this section we derive an explosion-type condition. 
Let \((m_n)_{n \in \mathbb{N}}\) be an increasing sequence of non-negative real numbers. 
We define
\begin{align}\label{explosion sequence}
\rho_n := \inf(t \geq 0 : \|X_t\| \geq m_n \textup{ or } \|X_{t-}\|\geq m_n) \wedge n,
\end{align}
and impose the following condition:
\begin{SA}\label{SA3}
Assume that for all \(n \in \mathbb{N}\) we have \(\rho_n \leq \lim_{m \to \infty} \tau_m\), where \((\tau_n)_{n \in \mathbb{N}}\) is an arbitrary announcing sequence for \(A\), and 
that \((\rho_n)_{n \in \mathbb{N}}\) satisfies Convention \ref{conv}.
\end{SA}
\begin{remark}\label{local character remark}
If for all \(n \in \mathbb{N}\) there exists a positive constant \(c_n\) such that 
\begin{align*}
\frac{\1_{\{Z_- > 0\}}}{Z^2_-} \cdot C(Z)^p_{\rho_n} \leq c_n,
\end{align*}
where \(C(Z)\) is defined as in \eqref{C(Z)}, then the sequence \((\rho_n)_{n \in \mathbb{N}}\) satisfies Convention \ref{conv}, c.f. Lemma \ref{besser als novi lemma} in Appendix \ref{An Integrability Condition to Indentify}.
\end{remark}
Standing Assumption \ref{SA3} has a \emph{local character}, c.f. Remark \ref{local character remark}, and therefore a broad scope.
The following result gives a sufficient condition for the martingality of \(Z\). In the one-dimensional diffusion setting a similar result may be deduced from Theorem 2.1 in \cite{MU(2012)}.

\begin{corollary}\label{explosion cond coro}
\begin{enumerate}
\item[\textup{(i)}]
Assume part \textup{(I)} of Conditions \ref{cond main uni}.
If we have for all \(t \geq 0\) 
\begin{align}\label{explosion cond}
\lim_{n \to \infty} \q\bigg(\sup_{s \leq t} \|X_s\| < m_n\bigg) = 1,
\end{align}
then \(Z\) is an \((\F, \p)\)-martingale. If \(\p\)-a.s. \(\rho_n \uparrow_{n \to \infty} \infty\) and \(Z\) is an \((\F, \p)\)-martingale, then \eqref{explosion cond} holds.
\item[\textup{(ii)}]
Let \((\bar{B}(h), C, \bar{\nu})\) be a candidate triplet which coincides with \((\dot{B}(h), C, \dot{\nu})\) from \eqref{candidtate triplet main theorem}  on \(\bar{A}\).
If the SMP associated with \((\mathbb{H}; \of 0, \infty\of; \eta; X; \bar{B}(h), C, \bar{\nu})\) has a solution \(\q\) and satisfies \(\{\rho_n \wedge \tau_m, n, m \in \mathbb{N}\}\)-uniqueness and if \(m_n \uparrow_{n \to \infty} \infty\) then 
\(Z\) is an \((\F, \p)\)-martingale.
\end{enumerate}
\end{corollary}
\begin{proof}
(i). Due to Standing Assumption \ref{SA3}, part (II) of Conditions \ref{cond main uni} is satisfied.
Moreover, note that for all \(t \geq 0\) we have
\(
\{\rho_n > t\} = \{ \sup_{s \leq t} \|X_s\| < m_n\} \cap \{n > t\}.
\)
Therefore, Proposition \ref{main prop uni 1} implies that \eqref{explosion cond} is sufficient for \(Z\) to be an \((\F, \p)\)-martingale.
If \(Z\) is an \((\F, \p)\)-martingale and \(\p\)-a.s. \(\rho_n \uparrow_{n \to \infty} \infty\), then Lemma \ref{ruf mimic}, Lemma \ref{lemma 2} and Standing Assumption \ref{SA3} imply that \eqref{explosion cond} is satisfied.

(ii). Corollary \ref{coro subinterval} and Remark \ref{uni prop} yield that part (I) of Conditions \ref{cond main uni} is satisfied. Moreover, since \(m_n\uparrow_{n \to \infty} \infty\) and \(\q\)-a.e. path of \(X\) is in \(\mathbb{D}^\mathbb{H}\) we obtain that 
\begin{align*}
\lim_{n \to \infty} \q\left(\sup_{s \leq t} \|X_s\| < m_n\right) = \q\left(\sup_{s \leq t} \|X_s\| < \infty\right) = 1.
\end{align*}
Therefore part (i) yields the claim.
\end{proof}
Note that the condition \eqref{explosion cond} only depends on the process \(X\) under \(\q\). The condition may even imply that \(Z\) is a true martingale, while \(X\) explodes under \(\p\), c.f. Example 3.1 in \cite{MU(2012)}.
It is fascinating and surprising that such a condition 
not only holds in the well-studied one-dimensional diffusion setting, but also in our general infinite dimensional setting which allows for discontinuities.

In the case where \(A = \of 0, \infty\of\) and \((\Omega, \mathcal{F}, \F)\) is full we obtain a classical explosion condition which we state in the following corollary. 
\begin{corollary}
Assume that \(m_n \uparrow_{n \to \infty} \infty\), \(A = \of 0, \infty\of\), that \((\Omega, \mathcal{F}, \F)\) is full, and that the SMP \((\mathbb{H}; A; \eta; X; \dot{B}(h), C, \dot{\nu})\), where \(\dot{B}(h)\) and \(\dot{\nu}\) are given by \eqref{candidtate triplet main theorem} on \(A\), satisfies \(\{\rho_n \wedge \tau_m, n, m \in \mathbb{N}\}\)-uniqueness.
Then the following is equivalent:
\begin{enumerate}
\item[\textup{(i)}] \(Z\) is an \((\F, \p)\)-martingale.
\item[\textup{(ii)}] The SMP associated with \((\mathbb{H}; A; \eta; X; \dot{B}(h), C, \dot{\nu})\) has a solution.
\end{enumerate}
\end{corollary}
\begin{proof} The claim follows immediately from Theorem \ref{main theorem new 2} and Corollary \ref{explosion cond coro} (ii).
\end{proof}
\subsection{Martingality on Standard Systems}\label{Martingality on Standard Systems}
The following approach depends on topological properties of the underlying filtered space and is related to the setting of the F\"ollmer measure, c.f. \cite{follmer72,perkowski2015}.
In particular, the result applies if the underlying filtered space is given by \((\mathbb{D}^{E_\Delta}, \mathcal{D}^{E_\Delta}, \mathfrak{D}^{E_\Delta})\) where \(E\) is an arbitrary Polish space. Hence we even allow \(Z\) to be driven by \(E_\Delta\)-valued processes.
Connecting the approach to the SMP setting of the previous sections allows us to drop the uniqueness assumptions of Conditions ~\ref{cond main uni}.
For the terminology used in this section we refer to the Appendix \ref{Extension of Probability Measures}.
We recall the notation \(\rho = \lim_{n \to \infty} \rho_n\), where \((\rho_n)_{n \in \mathbb{N}}\) is given as in Convention ~\ref{conv}. 
\begin{proposition}\label{ext prop}
Assume that \((\Omega, \mathcal{F}_{\rho_n-})_{n \in \mathbb{N}}\) is a standard system. Then there exists a unique probability measure \(\q\) on \((\Omega, \mathcal{F}_{\rho-})\) such that for all \(n \in \mathbb{N}\), \(\q(F) = \E[Z_{\rho_n} \1_F]\) for all \(F \in \mathcal{F}_{\rho_n-}\).
If in addition \((\Omega, \mathcal{F}_{\infty-})\) is a standard Borel space and \(\mathcal{F}_{\rho-}\) is countably generated, then \(\q\) has an extension to \((\Omega, \mathcal{F}_{\infty-})\).
\end{proposition}
\begin{proof}
Denote \(\q_n := Z_{\rho_n} \cdot \p\) and note that for all \(A \in \mathcal{F}_{\rho_{n}-}\) we have 
\begin{align*}
\q_{n+1} (A) = \E[\E[Z_{\rho_{n+1}}\1_A|\mathcal{F}_{\rho_n}]] = \E[Z_{\rho_n}\1_{A}] = \q_n(A),
\end{align*}
where we use Doob's optional stopping theorem. 
Hence the first claim of the proposition is an immediate consequence of Parthasarathy's extension theorem, c.f. Theorem \ref{extension P} in Appendix \ref{Extension of Probability Measures}, and the identity \(\bigvee_{n \in \mathbb{N}} \mathcal{F}_{\rho_n-} = \mathcal{F}_{\rho-}\).
The second claim follows from Theorem \ref{extension P2} in Appendix \ref{Extension of Probability Measures}.
\end{proof}
As an immediate consequence of Proposition \ref{ext prop}, Lemma \ref{ruf mimic} and Lemma \ref{neues lemma} we obtain the following condition for the martingality of \(Z\).
\begin{proposition}\label{theorem extension1}
Assume that \((\Omega, \mathcal{F}_{\rho_n-})_{n \in \mathbb{N}}\) is a standard system and denote by \(\q\) the probability measure as given in Proposition \ref{ext prop}.
\begin{enumerate}
\item[\textup{(i)}]
If \eqref{Main condition} holds,
then \(Z\) is an \((\F, \p)\)-martingale. If \(\p\)-a.s. \(\rho_n \uparrow_{n \to \infty} \infty\) and \(Z\) is an \((\F, \p)\)-martingale, then \eqref{Main condition} holds.
\item[\textup{(ii)}]
If \eqref{Main condition ui} holds, 
then \(Z\) is a uniformly integrable \((\F, \p)\)-martingale.
\end{enumerate}
\end{proposition}
In general, due to the lack of information on \(\q\), these conditions are difficult to check. 
However, \eqref{Main condition} holds if \(\rho_n(\omega) \uparrow_{n \to \infty} \infty\) for all \(\omega \in \Omega\). This path-wise condition is stated in the following corollary, which generalizes Corollary 3.4 in \cite{RufSDE}.
\begin{corollary}\label{path assumption}
Assume that \((\Omega, \mathcal{F}_{\rho_n-})_{n \in \mathbb{N}}\) is a standard system.
If \(\{\rho_n > t\} \uparrow_{n \to \infty} \Omega\) for all \(t \geq 0\), then \(Z\) is an \((\F, \p)\)-martingale. If even \(\{\rho_n =\infty\} \uparrow_{n \to \infty} \Omega\), then \(Z\) is a uniformly integrable \((\F, \p)\)-martingale.
\end{corollary}
\begin{remark}
On the path-space \((\mathbb{C}^E, \mathcal{C}^E, \mathfrak{C}^E)\), the pathwise condition \(\{\rho_n > t\} \uparrow_{n \to \infty} \mathbb{C}^E\) from Corollary \ref{path assumption} is not sufficient for \(Z\) to be a \((\mathfrak{C}^E, \p)\)-martingale, c.f. \cite{RufSDE} for a counterexample.
\end{remark}
We now assume a setting comparable to the previous section, which enables us to derive structural properties of \(\q\).
These allow us to express \(\q\) in Proposition \ref{theorem extension1} as a solution to a SMP:
\begin{condition}\label{cond topo}
\begin{enumerate}
\item[\textup{(I)}]
Assume that \((\Omega, \mathcal{F}_{\rho_n-})_{n \in \mathbb{N}}\) is a standard system, that \(\mathcal{F}_{\rho -}\) is countably generated and that \((\Omega, \mathcal{F}_{\infty-})\) is a standard Borel space.
\item[\textup{(II)}]
Assume that \(\mathcal{F} = \mathcal{F}_{\infty - }\) and that for each \(n \in \mathbb{N}\) there exists an \(m_n > n\) such that 
\(\rho_n \wedge \tau_n < \rho_{m_n} \wedge \tau_{m_n}\) on \(\{\rho_n \wedge \tau_n< \infty\}\).
\item[\textup{(III)}]
Assume that \((\tau_n \wedge \rho_n)_{n \in \mathbb{N}}\) is a sequence of \(\F\)-predictable times, and that \(\F\) is quasi-left continuous.
\end{enumerate}
\end{condition}
\begin{theorem}\label{theorem extension2}
Assume Standing Assumption \ref{SA}, part  \textup{(I)} of Conditions \ref{cond topo}, and additionally part \textup{(II)} or \textup{(III)} of Conditions \ref{cond topo}. Moreover, denote by \(\q\) the extension of the probability measure as given in Proposition ~\ref{ext prop}.
Then 
\(\q\) is a solution to the SMP associated with \((\mathbb{H}; \bar{A}; \eta; X; \dot{B}(h), C, \dot{\nu})\), where \(\bar{A} = A \cap (\bigcup_{n \in \mathbb{N}} \of 0,  \rho_n\gs)\), and \(\dot{B}(h)\) and \(\dot{\nu}\) are given by \eqref{candidtate triplet main theorem}.
\end{theorem}
\begin{proof}
Lemma \ref{lemma stopped} yields that the probability measure \(\q_{n, n} = Z_{\tau_n \wedge \rho_n} \cdot \p\) is a solution to the SMP associated with \((\mathbb{H}; \tau_n \wedge \rho_n; \eta; X; \dot{B}(h), C, \dot{\nu})\).
Due to the construction of \(\q\) we have \(\q_{n, n} = \q\) on \(\mathcal{F}_{(\tau_n \wedge \rho_n)-}\). If we assume part (III) of Conditions \ref{cond topo}, then \(\q_{n, n} = \q\) on \(\mathcal{F}_{\tau_n \wedge \rho_n}\) which yields the claim.
Now assume part (II) of Conditions \ref{cond topo}. Following the proof of Proposition 3.1 in \cite{RufSDE}, we obtain for all \(F \in \mathcal{F}_{\tau_n \wedge \rho_n}\) 
\begin{align*}
\q(F) = \E[Z_{\tau_{m_n} \wedge \rho_{m_n}} \1_{F \cap \{\tau_n \wedge \rho_{n} < \infty\}}] + \E[Z_{\tau_n \wedge \rho_n} \1_{F \cap \{\tau_n \wedge \rho_n = \infty\}}] = \E[Z_{\tau_n \wedge \rho_n} \1_F],
\end{align*}
where we used \cite{DellacherieMeyer78}, Theorem V.56, and Doob's optional stopping theorem.
This finishes the proof.
\end{proof}
Combining Proposition \ref{theorem extension1} and Theorem \ref{theorem extension2} leads to a martingale condition similarly to the Propositions \ref{main prop uni 1} and \ref{main prop uni 2}. 
\begin{remark}
If the underlying filtered space is given by
\((\mathbb{D}^{\mathbb{H}_\Delta}, \mathcal{D}^{\mathbb{H}_\Delta}, \mathfrak{D}^{\mathbb{H}_\Delta}),
\) 
the topological assumptions of Theorem \ref{theorem extension2}, part (I) of Conditions \ref{cond topo}, are satisfied, c.f. Appendix \ref{Extension of Probability Measures}. 
\end{remark}

\section{A Case Study - Martingality in Terms of Infinite-Dimensional SDEs}
\label{A Case Study - Martingality in Terms of SDEs driven by Hilbert-Space-Valued Brownian Motion} 
The generality of our setting allows us to derive sufficient conditions for the martingality of stochastic exponentials driven by Hilbert-space-valued Brownian motion.
The main result of this section is in the spirit of Theorem 3.1.1 in \cite{BR16}.

Let us first introduce some additional terminology. 
As in the previous section let \(\mathbb{H}\) be a real separable Hilbert space. Moreover, we fix a filtered probability space \((\Omega, \mathcal{F}, \F, \p)\) with right-continuous filtration \(\F\) and a non-negative and self-adjoint nuclear operator \(Q\in \mathcal{N}(\mathbb{H}, \mathbb{H})\).
\begin{definition}
An \(\mathbb{H}\)-valued process \(W\) is called an \((\mathbb{H}, Q, \F, \p)\)-\emph{Brownian motion}, if the following holds:
\begin{enumerate}
\item[\textup{(i)}]
\(\p\)-a.s. \(W_0 = 0\).
\item[\textup{(ii)}]
\(W\) has \(\p\)-a.s. continuous paths.
\item[\textup{(iii)}]
For all \(s < t\) the random variable \(W_t - W_s\) is \(\p\)-independent of \(\mathcal{F}_s\);
\item[\textup{(iv)}]
For all \(s < t\) and all \(h \in \mathbb{H}\) the (real-valued) random variable \((W_t - W_s) \bullet h\) is centered Gaussian with variance \((t-s) (Qh \bullet h)\). 
\end{enumerate}
\end{definition}
	The operator \(Q\) is usually called \emph{covariance} of \(W\).
Moreover, as in the finite-dimensional case we have a L\'evy characterization theorem: An \(\mathbb{H}\)-valued continuous \(\F\)-adapted process \(W\) with \(\p\)-a.s. \(W_0 = 0\) is an \((\mathbb{H},Q, \F, \p)\)-Brownian motion if and only if it is a square-integrable \((\F, \p)\)-martingale with \(\lle W, W\rre^\p= I Q\), where \(I_t = t\) denotes the identity process, c.f. Proposition 4.11 in \cite{MP80} or 
Theorem 4.6 in \cite{deprato}.

It is also well-known that \(Q\) has a decomposition of the form \(Q = Q^{1/2} Q^{1/2}\), where \(Q^{1/2}\) is a non-negative and self-adjoint Hilbert-Schmidt operator. 

We denote by \(\mathbb{C}^\mathbb{H}_0\) the set of continuous functions \(\mathbb{R}^+ \to \mathbb{H}\) which start at \(0\).
The coordinate process is denoted by \(\widehat{X}\).
We equip \(\mathbb{C}^\mathbb{H}_0\) with the \(\sigma\)-field \(\mathcal{C}^\mathbb{H}_0 := \sigma(\widehat{X}_t, t \geq 0)\) and define the filtrations \(\mathfrak{C}^{\mathbb{H}, 0}_0 := (\mathcal{C}^{\mathbb{H}, 0}_{0, t})_{t \geq 0} := (\sigma(\widehat{X}_s, s \leq t))_{t \geq 0}\) and \(\mathfrak{C}^\mathbb{H}_0 := (\mathcal{C}^{\mathbb{H}}_{0, t})_{t \geq 0} := (\bigcap_{s > t} \mathcal{C}^{\mathbb{H}, 0}_{0, s})_{t \geq 0}\).

Let \(\p\) be a solution to the SMP on \((\mathbb{C}^\mathbb{H}_0, \mathcal{C}^\mathbb{H}_0, \mathfrak{C}^\mathbb{H}_0)\) associated with \((\mathbb{H}; \of 0, \infty\of; \epsilon_0; \widehat{X}; 0, IQ, 0)\). 
In other words, we assume that the coordinate process is an \((\mathbb{H}, Q, \mathfrak{C}^\mathbb{H}_0, \p)\)-Brownian motion.
Moreover, let \(\varphi : \mathbb{R}^+ \times \mathbb{C}^\mathbb{H}_0 \to \mathbb{H}\) be \(\mathfrak{C}^\mathbb{H}_0\)-predictable such that \(\varphi(\cdot, 0)\) is constant.
We use the notation \(\varphi(\cdot, \omega) =: \varphi_\cdot(\omega)\).
Next we formulate the main result of this section, which states that under a local Lipschitz and a linear growth condition 
the stochastic exponential
\begin{align*}
Z: = \mathcal{E}\left(\varphi(\widehat{X}) \cdot \widehat{X}\right)
\end{align*}
is a true martingale.
For details concerning stochastic integration w.r.t. Hilbert-space-valued processes we refer to the monographs of \cite{metivier,MP80}.
To clarify the notation of \(\varphi(\widehat{X})\cdot\widehat{X}\), we note that we identify \(\varphi\) with \(v \rightsquigarrow \varphi \bullet v\).
\begin{theorem}\label{existence uniqueness concrete}
Assume the following:
\begin{enumerate}
\item[\textup{(i)}]
For all \(\alpha \in (0, \infty)\) there exists a positive \cadlag increasing function \(L^\alpha\) such that
\begin{align*}
\qquad\quad\|Q\varphi(t, \omega) - Q\varphi(t, \omega^*)\|_{\mathbb{H}} \leq L^{\alpha}_t \sup_{s < t} \|\omega(s) - \omega^*(s)\|_\mathbb{H}
\end{align*}
for all \(t \geq 0\), and all \(\omega, \omega^* \in \{\bar{\omega} \in \mathbb{C}^{\mathbb{H}}_0 : \sup_{s < t} \|\bar{\omega}(s)\|_\mathbb{H} \leq \alpha\}\).
\item[\textup{(ii)}] There exists a constant \(\lambda > 0\) such that 
\begin{align*}
\|Q^{1/2} \varphi(t, \omega)\|^2_\mathbb{H} \leq \lambda\bigg(1 + \sup_{s < t} \|\omega(s)\|^2_\mathbb{H}\bigg),
\end{align*}
for all \((t, \omega) \in \mathbb{R}^+ \times \mathbb{C}^{\mathbb{H}}_0\).
\end{enumerate}
Then the process
\(
Z = \mathcal{E}(\varphi(\widehat{X}) \cdot \widehat{X}) 
\)
is an \((\mathfrak{C}^\mathbb{H}_0, \p)\)-martingale.
\end{theorem}
\subsection{Proof of Theorem \ref{existence uniqueness concrete}}\label{Proof of Theorem}
We assume the notation of Section \ref{MSE}. Then \(A' = A = \of 0, \infty\of\) and we may set \(\sigma_n = \tau_n = n\), and define
\begin{align}\label{rho neu}
\rho_n := \inf\big(t \geq 0 : \|\widehat{X}_t\| \geq n\big) \wedge n.
\end{align}
Let us first show that \(Z\) is a local \((\mathfrak{C}^{\mathbb{H}}_0, \p)\)-martingale with \((\mathfrak{C}^\mathbb{H}_0, \p)\)-localizing sequence \((\rho_n)_{n \in \mathbb{N}}\).
Due to Assumption (ii) in Theorem \ref{existence uniqueness concrete} we have identically
\begin{align*}
\left(\varphi(\widehat{X}) \bullet Q\varphi(\widehat{X})\right) \cdot I_{\rho_n} \leq \lambda n \left(1 + \sup_{s < \rho_n} \|\widehat{X}_s\|^2_\mathbb{H}\right) \leq \lambda n(1 + n^2).
\end{align*}
Hence Theorem 2.3 in \cite{gawarecki2010stochastic} yields that \(\varphi(\widehat{X}) \cdot \widehat{X}_{\cdot \wedge \rho_n}\) is a square-integrable \((\mathfrak{C}^\mathbb{H}_0, \p)\)-martingale, implying that \(Z\) is a local \((\mathfrak{C}^\mathbb{H}_0, \p)\)-martingale.
This furthermore implies that \((\rho_n)_{n \in \mathbb{N}}\) satisfies Convention \ref{conv}.
Now the claim of Theorem \ref{existence uniqueness concrete} is an immediate consequence of Corollary \ref{explosion cond coro} (ii) and
Lemma \ref{key1} below.
\begin{lemma}\label{key1}
Under the assumptions of Theorem \ref{existence uniqueness concrete}, 
the SMP on \((\mathbb{C}^\mathbb{H}_0, \mathcal{C}^\mathbb{H}_0, \mathfrak{C}^\mathbb{H}_0)\) associated with \((\mathbb{H}; A; \varepsilon_0; \widehat{X}; \dot{B}, IQ, 0)\), where
\begin{align}\label{BC}
\dot{B} = \frac{\1_{\{Z_- > 0\}}}{Z_-} \cdot \lle Z, M(h)\rre^\p, 
\end{align} 
has a solution and satisfies \(\{\rho_n \wedge \tau_m, n, m \in \mathbb{N}\}\)-uniqueness.
\end{lemma}
\begin{proof}
First we compute \(\lle Z, M(h)\rre^\p = \lle Z, \widehat{X}\rre^\p\). In view of Equation 4.1.4 in \cite{MP80}, \(\lle Z, \widehat{X}\rre^\p\) is the unique 
\(\mathfrak{C}^\mathbb{H}_0\)-predictable process of finite variation such that \(Z \widehat{X} - \lle Z, \widehat{X}\rre^\p\) is a local \((\mathfrak{C}^\mathbb{H}_0, \p)\)-martingale.
Since \(\mathbb{H}\) is assumed to be separable, weak and strong measurability are equivalent. Therefore \(Z \widehat{X} - \lle Z, \widehat{X}\rre^\p\) is an \(\mathbb{H}\)-valued local \((\mathfrak{C}_0^\mathbb{H}, \p)\)-martingale if and only if for all \(v \in \mathbb{H}\) the process \(Z \widehat{X} \bullet v - \lle Z, \widehat{X}\rre^\p \bullet v\) is an \(\mathbb{R}\)-valued local \((\mathfrak{C}^\mathbb{H}_0, \p)\)-martingale, c.f. \cite{gawarecki2010stochastic}, p. 21.
Define \(L^{\pm v} : \mathbb{H} \to \mathbb{R}\) by \(L^{\pm v} w = (\varphi(\widehat{X}) \pm v) \bullet w\), and note that the adjoint of \(L^{\pm v}\) is given by \(w \rightsquigarrow (\varphi(\widehat{X}) \pm v) w\).
Employing the polarization identity, the formula for the quadratic variation of Brownian integrals, c.f. Theorem 4.27 in \cite{deprato}, and Proposition A.2.2 in \cite{liu2015stochastic},
we obtain 
\begin{align*}
\langle Z, \widehat{X} \bullet v\rangle^\p 
&= Z_-/4 \cdot \big(\big\langle L^{+v} \cdot \widehat{X}, L^{+v}\cdot \widehat{X}\big\rangle^\p + \big\langle L^{-v} \cdot \widehat{X}, L^{-v} \cdot \widehat{X}\big\rangle^\p\big)
\\&=Z_-/4 \cdot \big( L^{+v} Q (\varphi(\widehat{X}) + v) \cdot I + L^{-v} Q (\varphi(\widehat{X}) - v)\cdot I\big)
\\&= \big(Z_- Q \varphi(\widehat{X}) \cdot I \big) \bullet v.
\end{align*}
Therefore we conclude \(\lle Z, M(h)\rre^\p = Z_- Q \varphi(\widehat{X}) \cdot I\), which implies that 
\begin{align*}
\frac{\1_{\{Z_- > 0\}}}{Z_-} \cdot \lle Z, M(h)\rre^\p = 
Q \varphi(\widehat{X}) \cdot I.
\end{align*}
Now we note that 
\begin{align*}
\|Q\varphi(t, \omega)\|_\mathbb{H}^2 \leq \|Q^{1/2}\|_{\textup{HS}}^2 \|Q^{1/2}\varphi(t, \omega)\|_\mathbb{H}^2 \leq \lambda \|Q^{1/2}\|^2_{\textup{HS}} \left(1 + \sup_{s < t}\|\omega(s)\|^2_\mathbb{H}\right),
\end{align*}
for all \((t, \omega)\in \mathbb{R}^+ \times \mathbb{C}^\mathbb{H}_0\), where \(\|\cdot\|_\textup{HS}\) denotes the Hilbert-Schmidt norm.
Hence we may apply Proposition \ref{theorem existence uniqueness SMP} in Appendix \ref{SMPs and SDEs} with \(F(t, \omega) = Q\varphi(t, \omega)\) and \(G(t, \omega) = \operatorname{Id}\), which yields 
that the SMP on \((\mathbb{C}^\mathbb{H}_0, \mathcal{C}^\mathbb{H}_0, \mathfrak{C}^\mathbb{H}_0)\) associated with \((\mathbb{H}; \of 0, \infty\of; \eta; \widehat{X}; \dot{B}, C, 0)\) 
has a solution and satisfies \(\mathcal{T}^*\)-uniqueness.
The continuity of \(\widehat{X}\) and 
Lemma 74.2 in \cite{RW} imply that for all \(n, m \in \mathbb{N}\), \(\rho_n \wedge \tau_m = \rho_n \wedge m\) are positive \(\mathfrak{C}^{\mathbb{H}, 0}_0\)-stopping times, i.e. \(\rho_n \wedge \tau_m \in \mathcal{T}^*\).
This implies the claim.
\end{proof}
\appendix
\section{Remarks on the \emph{Usual Conditions}}\label{A UH}
Girsanov's theorem is at the core of our proofs. As pointed out by \cite{Bichteler02} on the infinite time horizon some formulations of Girsanov's theorem may be false if the underlying filtered space is complete. 
Our framework therefore requires incompleteness of filtrations.
Since the usual conditions, and therewith completeness  of filtrations, are frequently imposed in the literature we summarize some results which allow us to lift results from completed filtrations to uncompleted filtrations.

Let \((\Omega, \mathcal{F}, \F, \p)\) be a filtered probability space with right-continuous filtration \(\F\), and let \(\mathbb{H}\) be a real separable Hilbert space. This notation stays in force throughout all appendices.
We denote \(\mathcal{F}^{\hspace{0.03cm}\p}\) the \(\p\)-completion of the \(\sigma\)-field \(\mathcal{F}\) and \(\F^{\hspace{0.02cm}\p}\) the \(\p\)-completion of the filtration \(\F\), i.e. \(\F^{\hspace{0.02cm}\p} = (\mathcal{F}^{\hspace{0.03cm}\p}_t)_{t \geq 0}\).
\begin{lemma}[Lemma I.1.19 in \cite{JS}]\label{ST comp}
Any \(\F^\p\)-stopping time is \(\p\)-a.s. equal to an \(\F\)-stopping time.
\end{lemma}
\begin{lemma}[Lemma I.2.17 in \cite{JS}, Lemma A.2 in \cite{perkowski2015}]\label{P comp}
Any \(\F^\p\)-predictable, resp. \(\F^\p\)-optional, process is \(\p\)-indistinguishable from an \(\F\)-predictable, resp. \(\F\)-optional, process.
\end{lemma}
\begin{lemma}\label{LM comp}
An \(\mathbb{H}\)-valued \(\F\)-adapted process \(Z\) is a local \((\F, \p)\)-martingale if and only if it is a local \((\F^\p, \p)\)-martingale.
Moreover, each local \((\F^\p, \p)\)-martingale is \(\p\)-indistinguishable from a local \((\F, \p)\)-martingale.
\end{lemma}
\begin{proof}
Due to Lemma \ref{ST comp}, for the first claim it is sufficient to show that \(Z\) is an \((\F, \p)\)-martingale if and only if it is an \((\F^\p, \p)\)-martingale.
This follows immediately from the fact that for all \(t \geq 0\) any set in \(\mathcal{F}^\p_t\) are \(\p\)-a.s. equal to a set in \(\mathcal{F}_t\).
The second claim follows from the first one and Lemma \ref{P comp}.
\end{proof}
\section{Extensions of non-negative local martingales}\label{Extensions of non-negative local martingales on sets of interval type}
As pointed out at the end of Section \ref{Processes on Stochastic Intervals}, \(X\in \mathcal{C}^A\) does not imply that the extension \(\widetilde{X}\in \mathcal{C}\). In this appendix we show that this claim, however, is true if \(X\) is non-negative and \(\mathcal{C}\) is the class of real-valued local martingales. 
\begin{proposition}\label{lemma tilde Z local martingale}
For \(A \in \mathcal{I}(\F)\) let \(X \in \mathcal{M}_\textup{loc}^{A}(\mathbb{R},\F, \p)\) be non-negative, then 
we have \(\widetilde{X} \in \mathcal{M}_{\textup{loc}}(\mathbb{R},\F, \p)\), where \(\widetilde{X}\) is defined as in \eqref{extension}.
\end{proposition}
\begin{proof}
Let \((\tau_n)_{n \in \mathbb{N}}\) be an announcing sequence for \(A\). 
Then \(X^{\tau_n}\) is a non-negative \((\F, \p)\)-supermartingale and 
Lemma 5.17 in \cite{J79} yields that \(\widetilde{X}\) is a non-negative \((\F, \p)\)-supermartin\-gale. 
We follow main parts of the proof of \cite{J79}, Lemma 12.43.
Since \(\widetilde{X}\) is a non-negative \((\F, \p)\)-supermartingale, the Doob-Meyer decomposition of \(\widetilde{X}\), c.f. \cite{LS}, Theorem 3.9, is given by
\begin{align*} 
\widetilde{X} = M - U,\quad M \in \mathcal{M}_\textup{loc}(\mathbb{R}, \F, \p),\ U \in \mathcal{P}(\F) \cap \mathcal{A}^+(\F, \p).
\end{align*}
Since \(\widetilde{X}^{\tau_n}\in \mathcal{M}_\textup{loc}(\mathbb{R}, \F, \p)\) it follows from \cite{JS}, Corollary I.3.16, that \(U^{\tau_n} = 0\). 
Hence \(U = 0\) on \(\of 0, \tau\of \cup (\of \tau, \infty\of \cap (G^c \times \mathbb{R}^+))\), where \(G = \bigcap_{n \in \mathbb{N}} \{\tau_n < \tau < \infty\}\).
Moreover, since \(\widetilde{X} = \widetilde{X}^\tau\), we have that \(U = \Delta U_\tau \1_{\of \tau, \infty\of \cap (G \times \mathbb{R}^+)}\).
We obtain from \cite{JS}, Corollary I.2.31 and Proposition I.2.6, that 
\begin{align}\label{pred proj jump}
^p(\Delta \widetilde{X}) =\ ^p(\Delta M) - \ ^p(\Delta U) =-\ ^p(\Delta U) = -\Delta U.
\end{align}
Next we define the \(\F\)-stopping time
\begin{align*}
\bar{\tau} := (\tau)_G :=
 \begin{cases} \tau& \textup{ on } G,\\ + \infty&\textup{ on } G^c.
\end{cases}
\end{align*}
By construction we have on \(G = \{\bar{\tau} < \infty\}\) that \(\bar{\tau} = \tau\) and 
\begin{align}\label{jump vanish}
\Delta \widetilde{X}_{\bar{\tau}} = \widetilde{X}_{\tau} - \widetilde{X}_{\tau-} = 0.
\end{align}
Note that the sequence \(((\tau_n)_{\{\tau_n < \tau\}} \wedge n)_{n \in \mathbb{N}}\) of \(\F\)-stopping times foretells \(\bar{\tau}\) in the sense of \cite{DellacherieMeyer78}, Definition IV.70. Since \(\{\bar{\tau} = 0\} = \emptyset\), Theorem IV.71 in \cite{DellacherieMeyer78} yields that \(\bar{\tau}\) is an \(\F\)-predictable time.
Therefore, thanks to \eqref{pred proj jump}, \eqref{jump vanish} and Theorem I.2.28 in \cite{JS}, we obtain that on \(G\) we have \(\p\)-a.s.
\begin{align}\label{jacod conclusion}
- \Delta U_{\tau} = - \Delta U_{\bar{\tau}} = \E[\Delta \widetilde{X}_{\bar{\tau}} |\mathcal{F}_{\bar{\tau}-}] = 0.
\end{align}
This yields \(U = 0\) 
and hence \(\widetilde{X} = M \in \mathcal{M}_\textup{loc}(\mathbb{R},\mathfrak{F}, \p)\)
, which finishes the proof.
\end{proof}

\section{On Hilbert-space-valued Semimartingales 
}\label{A HVS}
Hilbert-space-valued semimartingales have been studied since the 1980s, c.f. \cite{,metivier,MP80}. 
In this appendix we define Hilbert-space-valued semimartingales, recall the most important results and transfer them to the setting of incomplete filtrations. Moreover, we generalize the Girsanov theorem for semimartingales as given in \cite{JS}, Theorem III.3.24, to the infinite-dimensional setting. We also give the most important structural results for Hilbert-space valued semimartingales on stochastic intervals as defined in Section \ref{Hilbert-Vlaued Semimartingales and Semimartingale Problems on Stochastic Intervals}.
The definitions and most of the proofs are similar to the real-valued case.

Let \(\mathbb{H}\) be a real separable Hilbert space, 
denote its scalar product by \(\bullet\) and its induced norm by \(\|\cdot\|\).
Thanks to the separability of \(\mathbb{H}\), strong measurability is equivalent to weak measurability.\footnote{\(f\) is strongly measurable if the real-valued function \(f \bullet h\) is measurable (in the ordinary sense) for all \(h\in \mathbb{H}\), c.f. \cite{curtain2012introduction}, Lemma A.5.2}
\begin{definition}
We call a \cadlag \(\mathbb{H}\)-valued process \(X\) an \emph{\((\F, \p)\)-semimartingale}, if 
\begin{align*}
X = X_0 + M + V,
\end{align*}
where \(X_0\) is \(\mathcal{F}_0\)-measurable and \(\mathbb{H}\)-valued, \(M\) is an \(\mathbb{H}\)-valued local \((\F, \p)\)-martingale with \(M_0 = 0\), and \(V\) is an \(\F\)-adapted \cadlag \(\mathbb{H}\)-valued process of finite variation with \(V_0 = 0\). 
We call \(X\) a \emph{special \((\F, \p)\)-semimartingale}, if \(V\) is \(\F\)-predictable. 
\end{definition}
In the next section we show that many usual facts on \(\mathbb{R}^d\)-valued semimartingales also hold for \(\mathbb{H}\)-valued semimartingales. 
Although the monographs \cite{,metivier,MP80} usually assume the usual conditions, in view of Lemma \ref{P comp} and Lemma \ref{LM comp} the claims can readily be transferred to an uncompleted filtration
\subsection{Decompositions of Hilbert-space-valued Semimartingales}
Let \(X \in \mathcal{S}(\mathbb{H}, \F, \p)\).
We first collect important basic assertions on the semimartingale decomposition.
\begin{proposition}[Exercise 5.11 in \cite{metivier}]\label{bj special}
If \(\|\Delta X\|\) is bounded, then \(X\) is a special \((\F, \p)\)-semimartingale. 
\end{proposition}
\begin{proposition}\label{pred finie var 0}
All \(\mathbb{H}\)-valued \(\F\)-predictable local \((\F, \p)\)-martingales of finite variation are \(0\) up to \(\p\)-indistinguishability.
\end{proposition}
\begin{proof}
Due to Exercise 5.13 in \cite{metivier} all locally square-integrable \((\F, \p)\)-martingales of finite \(\p\)-variation are \(0\) up to \(\p\)-indistinguishability. Due to Corollary 13.8 in \cite{metivier} all \(\F\)-predictable local \((\F, \p)\)-martingales are continuous and hence locally \((\F, \p)\)-square integrable. This yields the claim. 
\end{proof}
We call two square-integrable \((\F, \p)\)-martingales \(N, M\) to be \((\F, \p)\)-\emph{strongly orthogonal}, if for all \(\F\)-stopping times \(\tau\) it holds that \(\E[(N \bullet M)_\tau] = 0\).
This definition extends readily to locally square-integrable \((\F, \p)\)-martingales.
The following is a consequence of Theorem ~20.2, Theorem 23.6 and Exercise 5.13 in \cite{metivier}.
\begin{proposition}\label{decomposition clomp}
Each \(\mathbb{H}\)-valued \((\F, \p)\)-semimartingale \(X\) has a decomposition 
\begin{align*}
X = X_0 + X^c + M^d + V,
\end{align*}
where \(X_0\) is \(\mathcal{F}_0\)-measurable, \(X^c\) is a continuous local \((\F, \p)\)-martingale with \(X_0 = 0\), \(M^d\) is a locally square-integrable \((\F, \p)\)-martingale, strongly orthogonal to \(X^c\), with \(M^d_0 = 0\), and \(V\) is of finite variation with \(V_0 = 0\).
Moreover, \(X^c\) is unique (up to \(\p\)-indistinguishability), and called \textup{the continuous local \((\F, \p)\)-martingale part of \(X\)}.
\end{proposition}
\begin{proposition}\label{decomp cont SM}
Assume that \(X\) has \((\F, \p)\)-characteristics \((B(h), C, 0)\). Then \(B(h) = B\), and
\(
X = X_0 + B + X^c.
\)
\end{proposition}
\begin{proof}
Since \(\nu = 0\), it follows that \(\p\big(\mu^X(\mathbb{R}^+ \times \mathbb{H}) \geq 1\big) \leq \E[\mu^X(\mathbb{R}^+ \times \mathbb{H})] = 0,\) which implies \(\Delta X = 0\). The uniqueness 
of \(X^c\) yields the claim. 
\end{proof}
\subsection{Semimartingale Characterstics}\label{Semimartingale Characterstics}
Our next step is to define semimartingale characteristics of Hilbert-space-valued semimartingales, similar to \cite{J79,JS}, c.f. also \cite{Xie}. 

In the following we denote by \(\| \cdot\|_1\) the \emph{trace-norm} on \(\mathbb{H} \otimes \mathbb{H}\), and by \(\mathbb{H}\hspace{0.07cm} \hat{\otimes}_1 \mathbb{H}\) the completion of \(\mathbb{H} \otimes \mathbb{H}\) w.r.t. \(\|\cdot\|_1\), c.f. \cite{metivier2006reelle}, Section 11.
Let \(M\) be an \(\mathbb{H}\)-valued square-integrable \((\F, \p)\)-martingale with \(M_0 = 0\). Due to Theorem 14.3 in \cite{MP80} there exists a unique \cadlag
\(\F\)-predictable \(\mathbb{H}\hspace{0.07cm} \hat{\otimes}_1 \mathbb{H}\)-valued process \(\lle M, M\rre^\p\) of finite variation with \(\lle M, M\rre^\p_0 = 0\), 
such that \(M\otimes M - \lle M, M\rre^\p\) is an \(\mathbb{H} \hspace{0.07cm}\hat{\otimes}_1\mathbb{H}\)-valued \((\F, \p)\)-martingale.

In the more general case where \(M\) is a locally square-integrable \((\F, \p)\)-martingale with \((\F, \p)\)-localizing sequence \((\tau_n)_{n \in \mathbb{N}}\) we may set
\begin{align}\label{local tensor QV}
\lle M, M\rre^\p := \sum_{n \geq 1} \lle M^{\tau_n}, M^{\tau_n}\rre^\p \1_{\gs \tau_{n-1}, \tau_n\gs},
\end{align}
with the convention \(\tau_0 := 0\). 
Clearly, \(\lle M^{\tau_n}, M^{\tau_n}\rre^\p_{\cdot \wedge \tau_m} = \lle M^{\tau_m}, M^{\tau_m}\rre^\p\) for \(m \leq n\), due to uniqueness.
Therefore 
\begin{align*}
M \otimes  M - \lle M, M\rre^\p \in \mathcal{M}_\textup{loc}(\mathbb{H}\hspace{0.07cm} \hat{\otimes}_1 \mathbb{H}, \F, \p),
\end{align*}
and Proposition \ref{pred finie var 0} 
yields that \(\lle M, M\rre^\p\) is the unique \cadlag \(\F\)-predictable 
process of finite variation such that \(M \otimes M - \lle M, M\rre^\p\) is an \(\mathbb{H}\hspace{0.07cm} \hat{\otimes}_1\mathbb{H}\)-valued local \((\F, \p)\)-martingale.
Thanks to the uniqueness, the definition \eqref{local tensor QV} is independent of the localizing sequence.
It is convenient to associate to each element \(b\) in \(\mathbb{H}\hspace{0.07cm} \hat{\otimes}_1 \mathbb{H}\) a \emph{nuclear operator} \(\tilde{b} : \mathbb{H} \to \mathbb{H}\) which is uniquely defined by 
\begin{align*}
\tilde{b}h \bullet g = b \bullet_2 (h \otimes g),\quad (h, g) \in \mathbb{H}\times \mathbb{H},
\end{align*}
where \(\bullet_2\) denotes the usual scalar product \((h_1 \otimes g_1)\bullet_2 (h_2 \otimes g_2) = (h_1 \bullet h_2) (g_1 \bullet g_2)\), \(h_i, g_i \in \mathbb{H}\), c.f. \cite{metivier2006reelle}, Section 11. 
With a slight abuse of notation we denote the nuclear operator associated to \(\lle M,M\rre^\p\) again by \(\lle M,M\rre^\p\).


Now let \(X\) be an \(\mathbb{H}\)-valued \((\F, \p)\)-semimartingale.
We associate to \(X\) the integer-valued random measure 
\begin{align*}
\mu^X(\omega, \dd t, \dd x) := \sum_{s > 0} \1_{\{\|\Delta X_s\| \not = 0\}} \varepsilon_{(s, \Delta X_s)} (\dd t, \dd x),
\end{align*}
and denote by \(\nu^{X, \p}\), or simply by \(\nu\), the \((\F, \p)\)-compensator of \(\mu^X\) defined as in \cite{JS}, Theorem II.1.8.

Next we define the modified \(\mathbb{H}\)-valued semimartingale \(X(h)\) by 
\begin{align}\label{X(h)}
X(h) := X - \sum_{s \leq \cdot} (\Delta X_s - h(\Delta X_s)),
\end{align}
where \(h : \mathbb{H} \to \mathbb{H}\) is a \emph{truncation function}, i.e. a bounded function with \(h(x) = x\) on \(\{x \in \mathbb{H} : \|x\| \leq \epsilon\}\) for some \(\epsilon > 0\).
We note that the sum \(\sum_{s \leq \cdot} (\Delta X_s - h(\Delta X_s))\) is well-defined due to the \cadlag paths of \(X\).
Since \(X(h)\) has obviously bounded jumps, \(X(h)\) is a special \((\F, \p)\)-semimartingale, c.f. Proposition \ref{bj special} in Appendix \ref{A HVS}, and hence admits a unique 
decomposition 
\begin{align}\label{B(h) def}
X(h) = X_0 + M(h) + B(h), 
\end{align}
where \(X_0\) is \(\mathcal{F}_0\)-measurable, \(M(h)\) is a local \((\F, \p)\)-martingale with \(M(h)_0 = 0\), and \(B(h)\) is a \(\F\)-predictable process of finite variation with \(B(h)_0 = 0\).
We finally are in the position to define the semimartingale characteristics of \(X\):
\begin{definition}
The \((\F, \p)\)-\emph{semimartingale characteristics of \(X\)} are given by \((B(h), C, \nu)\) consisting of 
\begin{enumerate}
\item[\textup{(i)}]
the \cadlag \(\F\)-predictable process \(B(h)\) of finite variation as given in \eqref{B(h) def}.
\item[\textup{(ii)}]
\(C = \lle X^c, X^c\rre^\p\), where \(X^c\) is the continuous local \((\F, \p)\)-martingale part of \(X\).\footnote{\(\lle X^c, X^c\rre^\p\) is obviously well-defined since all continuous local \((\F, \p)\)-martingales are locally square-integrable \((\F, \p)\)-martingales}
\item[\textup{(iii)}]
\(\nu\) is the \((\F, \p)\)-compensator of \(\mu^X\).
\end{enumerate}
\end{definition}
\begin{remark}
For each truncation function \(h\), the characteristics are unique. 
\end{remark}

\subsection{Girsanov's Theorem for Hilbert-space-valued Semimartingales}\label{Girsanov's Theorem for Hilbert-Space-valued Semimartingales}
In this appendix we give a generalization of Girsanov's theorem for semimartingales as stated in \cite{JS}, Theorem III.3.24, to Hilbert-space-valued semimartingales.
Since this result is fundamental for our main assertions we provide a detailed proof.

We introduce additional notation.
Let \(X\in \mathcal{S}(\mathbb{H}, \F, \p)\) and let \(\mathbb{K}\) be a further real separable Hilbert space with \(Y \in \mathcal{S}(\mathbb{K}, \F, \p)\). 
Moreover, denote \(\Pi_n := \{0 \leq t_0 < ... < t_k < ...\}\) a partition of \(\mathbb{R}^+\) such that \(\lim_{n \to \infty} \sup_{t_i \in \Pi_n} |t_{i+1} - t_i| = 0\).
Then, c.f. Section 4.1 in \cite{MP80}, for every \(t \geq 0\) the following limit exists
\begin{align}\label{TQV}
\llbr X, Y\rrbr^\p_t := \underset{n \to \infty}{\p\ \textup{-}\lim} \sum_{t_i \in \Pi_n} \big(X_{t_{i+1}\wedge t} - X_{t_{i} \wedge t}\big) \otimes \big(Y_{t_{i+1}\wedge t} - Y_{t_i \wedge t}\big).
\end{align}
The process \(\llbr X, Y\rrbr^\p\) is called \emph{tensor quadratic variation of \(X\) and \(Y\)}.
Similar to \cite{JS}, Section III.3c), we denote by \(M^{\p}_{\mu^X}(Z |\widetilde{\mathcal{P}}(\F))\) the \(\mathcal{P}(\F) \otimes \mathcal{H} =: \widetilde{\mathcal{P}}(\F)\)-conditional \(M^{\p}_{\mu^X}\)-expectation, where 
\begin{align*}
M^{\p}_{\mu^X} (W) = \E[W * \mu^X_\infty],
\end{align*}
for \(\widetilde{\mathcal{P}}(\F)\)-measurable functions \(W\).\footnote{this terminology has its origin in the Russian literature where \(M\) is used for the (mathematical) expectation}

\begin{theorem}[Girsanov's Theorem]\label{GT}
Let \(\q \ll_\textup{loc} \p\) with density process \(Z\), and \(X\) be an \(\mathbb{H}\)-valued \((\F, \p)\)-semimartingale with \((\F, \p)\)-characteristics \((B(h), C, \nu)\).
Moreover, let \(M(h)\) be as in \eqref{B(h) def}. 
Then there exists a unique \cadlag \(\F\)-predictable process \(\lle Z, M(h)\rre^\p\) of finite variation such that \(ZM(h) - \lle Z, M(h)\rre^\p\) is a local \((\F, \p)\)-martingale. Moreover, \(X\) is an \((\F, \q)\)-semimartingale such that its \((\F, \q)\)-characteristics are given by 
\begin{align*}
B'(h) &= B(h) + \frac{\1_{\{Z_- > 0\}} }{Z_-}\cdot \lle Z, M(h)\rre^\p,\   C' = C\textit{ and } \nu'(\dd t, \dd x) = Y(t, x) \nu(\dd t, \dd x),
\end{align*}
where \(Y =\1_{\{Z_- > 0\}}/Z_-
 M^\p_{\mu^X}(Z |\widetilde{\mathcal{P}}(\F)).\)
\end{theorem}
\begin{proof}
We mimic the proofs of Theorem III.3.11 and Lemma III.3.14 in \cite{JS}, and Theorem 4.7 in \cite{MP80}.
The theorem on p.208 in \cite{metivier} yields that \(X\) is an \((\F, \q)\)-semimartingale.
We define \(\tau_n = \inf(t \geq 0 : Z_t < 1/n)\) and note that \(1/Z_-\) is well-defined on \(\llbr 0, \tau_n\rrbr\).
Recall that \(\|\Delta M(h)\|\) is bounded by some constant \(m\), due to construction. Moreover, due to Corollary I.4.55 in \cite{JS}, \(U = (\sum_{t \geq 0} (\Delta Z_t)^2)^{1/2}\) is \(\F\)-locally \(\p\)-integrable. Denote an \((\F, \p)\)-localizing sequence by \((\rho_n)_{n \in \mathbb{N}}\) and set \(A := \llbr Z, M(h)\rrbr^\p\). 
Define \(\gamma_n := \inf(t \geq 0 : \|A_t\| \geq n)\) and note that \(\p\)-a.s. \(\gamma_n \uparrow_{n \to \infty} \infty\).
Since
\begin{align*}
\|A_{t \wedge \rho_n \wedge \gamma_n}\| \leq n + \|\Delta A_{t \wedge \rho_n \wedge \gamma_n}\| \leq n + m |\Delta Z_{t \wedge \rho_n \wedge \gamma_n}| \leq n + m U_{\rho_n},
\end{align*}
the stopped process \(A^{\rho_n \wedge \gamma_n}\) is an \((\F, \p)\)-quasi-martingale of class \([D]\), c.f. Definitions 8.6 and 13.1 in \cite{metivier}. 
Theorem 13.3 in \cite{metivier} and Proposition 9.14 in \cite{MP80} imply that a unique \cadlag 
\(\F\)-predictable process \(\lle Z, M(h)\rre^\p\) of finite variation exists, such that \(A - \lle Z, M(h)\rre^\p\) is a local \((\F, \p)\)-martingale.
We define an \(\F\)-predictable process \(G\) of finite variation by
\begin{align*}
G = \frac{\1_{\{Z_- > 0\}}}{Z_-} \cdot \lle Z, M(h)\rre^\p.
\end{align*}
Now, due to Equation 4.1.4 in \cite{MP80}, we have
\begin{align*}
Z^{\tau_n} G^{\tau_n} &= Z_-\1_{\of 0, \tau_n\gs} \cdot G + G_-\1_{\of 0, \tau_n\gs} \cdot Z + \llbr Z, G\rrbr^\p_{\cdot \wedge\tau_n}
\\&= \lle Z, M(h)\rre^\p_{\cdot \wedge \tau_n} + G \1_{\of 0, \tau_n\gs}\cdot Z,
\end{align*}
where we use that, due to the fact that \(G\) is \cadlag\hspace{-0.15cm}, \(\F\)-predictable and of finite variation, 
\begin{align*}
\llbr Z, G\rrbr^\p_{\cdot \wedge \tau_n} 
= (G- G_-)\1_{\of 0, \tau_n\gs} \cdot Z,
\end{align*}
c.f. the proof of Proposition I.4.49 in \cite{JS}.
Hence 
the process \(Z^{\tau_n} G^{\tau_n} - \lle Z, M(h)\rre^\p_{\cdot \wedge \tau_n}\) is a local \((\F, \p)\)-martingale. 
Therefore we obtain that the process
\begin{align*}
( Z^{\tau_n}M(h)^{\tau_n} &- \lle Z, M(h)\rre^\p_{\cdot \wedge\tau_n}) - (Z^{\tau_n} G^{\tau_n} - \lle Z, M(h)\rre^\p_{\cdot \wedge\tau_n}) 
= Z^{\tau_n}(M(h)^{\tau_n} - G^{\tau_n}) 
\end{align*}
is a local \((\F, \p)\)-martingale.
It follows similarly to the proof of Proposition III.3.8 (ii) in \cite{JS} that \(M(h)^{\tau_n} - G^{\tau_n}\) is a local \((\F, \q)\)-martingale. 
Since the class of local martingales is stable under localization and \(\q\)-a.s. \(\tau_n \uparrow_{n \to \infty} \infty\), which holds due to the fact that \(\q(Z_t = 0) = \E[Z_t \1_{\{Z_t = 0\}}] = 0\) for all \(t \geq 0\), we conclude that \(M(h) - G\) is a local \((\F, \q)\)-martingale. 
Hence, in view of \eqref{B(h) def}, the representation of \(B'(h)\) follows from Proposition \ref{pred finie var 0}.
A similar argumentation as above yields that \(X^c + \1_{\{Z_- > 0\}}/Z_- \cdot \lle Z^c, X^c\rre^\p\) is a continuous local \((\F, \q)\)-martingale, where \(\lle Z^c, X^c\rre^\p\) is the unique 
continuous \(\F\)-predictable process of finite variation such that \(Z^c X^c - \lle Z^c, X^c\rre^\p\) is a local \((\F, \p)\)-martingale.
Hence, Proposition 4.3 in \cite{MP80} implies
 \(\llbr X^{c, \q}, X^{c, \q} \rrbr^\q = \llbr X^{c}, X^{c}\rrbr^\q\),
where \(X^{c, \q}\) denotes the continuous local \((\F, \q)\)-martingale part of \(X\), and \(X^{c}\) denotes the continuous local \((\F, \p)\)-martingale part of \(X\).
The assumption that \(\q \ll_\textup{loc} \p\) implies readily that \(\llbr X^{c}, X^{c}\rrbr^\q = \llbr X^c, X^c\rrbr^\p\). 
Due to  Theorem 22.8 in \cite{metivier}, 
it holds that \(\llbr X^c, X^c\rrbr^\p = \lle X^c, X^c\rre^\p = C\).
The last remaining claim concerns the \((\F, \q)\)-compensator of the jump measure associated with \(X\).
This claim follows from an application of the Girsanov-type theorem for random-measures as given in \cite{JS}, Theorem III.3.17. This finishes the proof.
\end{proof}

\begin{remark}
In the case where \(\mathbb{H} = \mathbb{R}^d\), Theorem \ref{GT} boils down to the classical version of Girsanov's Theorem for semimartingales given in \cite{JS}, Theorem III.3.24. 
\end{remark}


\subsection{Hilbert-space-valued Semimartingales on Stochastic Intervals}\label{facts HSVSM SI}
Next we collect the most important consistency facts on semimartingales on stochastic intervals.
\begin{proposition}\label{exist A c}
Let \(X \in \mathcal{S}^A(\mathbb{H}, \F, \p)\). Then the following holds:
\begin{enumerate}
\item[\textup{(i)}]
The \((\F, \p, A)\)-characteristics of \(X\) exist and are unique. 
\item[\textup{(ii)}]
Let  \((B(h), C, \nu)\) be \((\F, \p, A)\)-characteristics of \(X\), and let \(\rho\) be an \(\F\)-stopping time such that
\(\of 0, \rho \gs\subseteq A\). 
Then the \((\F, \p)\)-characteristics of \(X^\rho\in \mathcal{S}(\mathbb{H}, \F, \p)\) are \(\p\)-indistinguishable from \((B^\rho(h), C^\rho, \nu^\rho)\).
\item[\textup{(iii)}]
\((B(h), C, \nu)\) are the \((\F, \p, A)\)-characteristics of \(X\) if and only if for all \(\F\)-stopping times \(\rho\) such that \(\of 0, \rho\gs \subseteq A\) the triplet \((B^{\rho}(h), C^{\rho}, \nu^{\rho})\) is \(\p\)-indistinguishable of the \((\F, \p)\)-characteristics of \(X^{\rho}\). 
\end{enumerate}
\end{proposition}
\begin{proof}
(i). We first prove existence. Denote by \((\tau_n)_{n \in \mathbb{N}}\) an arbitrary announcing sequence for \(A\). Then for each \(n \in \mathbb{N}\) the \((\F, \p)\)-characteristics of the stopped process \(X^{\tau_n}\) exist. We denote them by \((B^n(h), C^n, \nu^n)\).
Now we \emph{stick them together} by setting
\begin{align*}
B(h) &= \sum_{n \geq 1} \1_{\gs \tau_{n-1}, \tau_n\gs}B^n(h), \ 
C = \sum_{n \geq 1} \1_{\gs \tau_{n-1}, \tau_n\gs} C^n\textup{ and }
\nu = \sum_{n \geq 1} \1_{\gs \tau_{n-1}, \tau_n\gs \times \mathbb{H}} \cdot \nu^n,
\end{align*}
with \(\tau_0 := 0\). 
Obviously, \(B(h), C\) and \(\nu\) are \(\F\)-predictable, and
\begin{equation}\label{cha eq}
\begin{split}
B(h)^{\tau_n} &= \sum_{i \leq n} \1_{\gs \tau_{i-1}, \tau_i\gs} B^i(h) + B^n(h)_{\tau_n} \1_{\gs \tau_n, \infty\of} = B^n(h),\\ 
C^{\tau_n} &= \sum_{i \leq n} \1_{\gs \tau_{i-1}, \tau_i\gs} C^i + C^n_{\tau_n} \1_{\gs \tau_n, \infty\of} = C^n,\\ 
\1_{\of 0, \tau_n\gs\times \mathbb{H}}\cdot\nu &= \sum_{i \leq n} \1_{\gs \tau_{i-1}, \tau_i\gs\times \mathbb{H}}\cdot \nu^i = \nu^n,
\end{split}
\end{equation}
where we used that \(B^n(h)_{\cdot \wedge \tau_i} = B^i(h), C^n_{\cdot \wedge \tau_i} = C^i\) and \(\1_{\of 0, \tau_i\gs \times \mathbb{H}} \cdot \nu^n = \nu^i\) for all \(i \leq n\). 
Now, in view of \eqref{cha eq}, it is obvious that the triplet \((B(h)^{\tau_n}, C^{\tau_n}, \nu^{\tau_n})\) is \(\p\)-indistinguishable from the \((\F, \p)\)-characteristics of \(X^{\tau_n}\), i.e. that \((B(h), C, \nu)\) are \((\F, \p,A)\)-characteristics of \(X\).
Once we have shown part (ii), the uniqueness of the \((\F, \p, A)\)-characteristics is a consequence of the uniqueness of the (ordinary) semimartingale characteristics and Proposition I.2.18 in \cite{JS}.

Therefore we now turn to the proof of (ii). Denote by \((\tau_n)_{n \in \mathbb{N}}\) an arbitrary announcing sequence for \(A\), and by \((\bar{B}(h), \bar{C}, \bar{\nu})\) the (ordinary) \((\F, \p)\)-characteristics of \(X^\rho\). Due to the uniqueness we have for all \(n \in \mathbb{N}\) 
\[(B^{\tau_n \wedge \rho}(h), C^{\tau_n \wedge \rho}, \nu^{\tau_n \wedge \rho}) = (\bar{B}^{\tau_n \wedge \rho}(h), \bar{C}^{\tau_n \wedge \rho}, \bar{\nu}^{\tau_n \wedge \rho}).\]
Now, since \(\of 0, \rho\gs \subseteq A\), letting \(n \to \infty\) and employing Proposition I.2.18 in \cite{JS} yields the claim.

Part (iii) is an immediate consequence of parts (i) and (ii).
\end{proof}
\begin{proposition}\label{clmp}
Let \(X \in \mathcal{S}^A(\mathbb{H}, \F, \p)\), then there exists a unique continuous local \((\F, \p)\)-martingale \(X^c\) on \(A\) such that for all \(\F\)-stopping times \(\rho\) with \(\of 0, \rho\gs\subseteq A\) the stopped process \(X^c_{\cdot \wedge \rho}\) is the continuous local \((\F, \p)\)-martingale part of \(X^\rho\).
\(X^c\) is called \emph{continuous local \((\F, \p)\)-martingale part of \(X\) on \(A\)}.
\end{proposition}
\begin{proof}
Denote by \((\tau_n)_{n \in \mathbb{N}}\) an arbitrary announcing sequence for \(A\).
Then let \(X^{c, n}\) be the unique continuous local \((\F, \p)\)-martingale part of \(X^{\tau_n}\), which exists due to Proposition \ref{decomposition clomp}.
We now \emph{stick them together} by setting
\begin{align*}
X^c := \sum_{n \in \mathbb{N}} X^{c, n} \1_{\gs \tau_{n-1}, \tau_n\gs},\textup{ where } \tau_0 := 0.
\end{align*}
Due to uniqueness we have \(X^{c}_{\cdot \wedge \tau_n} = X^{c, n}\), i.e. \(X^c\) is a continuous local \((\F, \p)\)-martingale on \(A\).
Now let \(\rho\) be an \(\F\)-stopping time with \(\of 0, \rho\gs\subseteq A\) and denote the continuous local \((\F, \p)\)-martingale part of \(X^\rho\) by \(\bar{X}\).
Uniqueness yields that we have for all \(\F\)-predictable times \(\gamma\) with \(\of 0, \gamma\gs\subseteq A\), \(\p\)-a.s. \(\bar{X}_{\tau_n \wedge \rho \wedge \gamma} = X^c_{\tau_n \wedge \rho \wedge \gamma}\) for all \(n \in \mathbb{N}\). Now letting \(n \to \infty\) and employing Proposition I.2.18 in \cite{JS} yields that \(\bar{X} = X^c_{\cdot \wedge \rho}\).
This finishes the proof.
\end{proof}
\section{An Integrability Condition to Indentify \((\rho_n)_{n \in \mathbb{N}}\)}\label{An Integrability Condition to Indentify}
This section provides an integrability condition to identify a sequence \((\rho_n)_{n \in \mathbb{N}}\) of stopping times which satisfy Convention \ref{conv}. 
Let \(\widetilde{Z}\) be a non-negative local \((\F, \p)\)-martingale and
set \(C(\widetilde{Z})\) as in \eqref{C(Z)}. 
\begin{lemma}\label{besser als novi lemma}
Let \(\sigma\) be an \(\F\)-stopping time. 
Assume that there exists a non-negative constant ~\(c\), such that 
\begin{align*} 
\frac{\1_{\{\widetilde{Z}_- > 0\}}}{\widetilde{Z}^2_-}\cdot C(\widetilde{Z})^p_{\sigma} \leq c. 
\end{align*}
Then \(\widetilde{Z}^\sigma\) is a uniformly integrable \((\F, \p)\)-martingale.
\end{lemma}
\begin{proof}
We first show that \(\p\)-a.s.
\begin{align}\label{J Z}
\frac{1}{\widetilde{Z}^2_-}\cdot C(\widetilde{Z})^p = \frac{\1_{\{\widetilde{Z}_- > 0\}}}{\widetilde{Z}^2_-} \cdot C(\widetilde{Z})^p.
\end{align}
Due to Lemma III.3.6 in \cite{JS} we have \(\p\)-a.s. \(\widetilde{Z}_t = 0\) for all \(t \geq \xi:= \inf(t \geq 0 : \widetilde{Z}_t = 0\textup{ or } \widetilde{Z}_{t-} = 0)\).
Due to Theorem I.3.18 in \cite{JS}, this implies 
\begin{align*}
\1_{\{\widetilde{Z}_- = 0\}} \cdot C(\widetilde{Z})^p = \big(\1_{\{\widetilde{Z}_- = 0\}} \cdot C(\widetilde{Z})\big)^p=\big(\1_{\{\widetilde{Z}_{\xi -} = 0\}}\1_{\of \xi\gs}\Delta C(\widetilde{Z})_\xi\big)^p= 0,
\end{align*}
which readily yields \eqref{J Z}.
Since 
\(\widetilde{Z}^{\sigma}\) is a non-negative local \((\F, \p)\)-martingale, 
the claim follows thanks to Theorem 8.25 in \cite{J79}. 
\end{proof}
Further localization conditions can be deduced from other integrability conditions in Chapter ~8 of the monograph of \cite{J79}.
\section{Classical Path-Spaces}\label{Classical Path-Spaces}
Particularly important underlying filtered spaces for SMPs are path spaces: Let \(E\) be a Polish space, then
the classical path spaces \((\mathbb{D}^{E}, \mathcal{D}^E, \mathfrak{D}^{E})\), respectively \((\mathbb{C}^E, \mathcal{C}^E, \mathfrak{C}^E)\), are defined as follows:
\(\mathbb{D}^{E}\), respectively \(\mathbb{C}^E\), is defined as the space of all c\`adl\`ag, respectively continuous, mappings \(\mathbb{R}^+ \to E\). \(\mathbb{D}^E\) equipped with the Skorokhod topology is a Polish space, c.f. \cite{EK}, Chapter 3. 
For \(\omega \in \mathbb{D}^E, \omega \in \mathbb{C}^E\) or \(\omega \in \mathbb{D}^{E_\Delta}\) we denote by \(\widehat{X}_t(\omega) = \omega(t)\) the so-called \emph{coordinate process}.
Denote by \(\mathcal{D}^E, \mathcal{C}^E, \mathcal{D}^{E_\Delta}\) respectively, the \(\sigma\)-fields generated by the coordinate process of \(\mathbb{D}^E, \mathbb{C}^E, \mathbb{D}^{E_\Delta}\) respectively.
Moreover, we define the filtrations \(\mathfrak{D}^{E, 0} = (\mathcal{D}^{E, 0}_t)_{t \geq 0}\) and \(\mathfrak{D}^{E} = (\mathcal{D}^E_t)_{t \geq 0} = (\bigcap_{s > t} \mathcal{D}^{E, 0}_s)_{t \geq 0}\) 
with
\(
\mathcal{D}^{E, 0}_t = \sigma(\widehat{X}_s, s \leq t)\) on \((\mathbb{D}^{E}, \mathcal{D}^E)\), \(\mathfrak{C}^{E, 0} = (\mathcal{C}^{E, 0}_t)_{t \geq 0}\) and \(\mathfrak{C}^E = (\mathcal{C}^E_t)_{t \geq 0} = (\bigcap_{s > t} \mathcal{C}^{E, 0}_s)_{t \geq 0}\) with \(\mathcal{C}^{E, 0}_t = \sigma(\widehat{X}_s, s \leq t)\) on \((\mathbb{C}^E, \mathcal{C}^{E})\), and analogously, we define 
\(\mathcal{D}^{E_\Delta, 0}_t = \sigma(\widehat{X}_s, s \leq t)\), \(\mathcal{D}^{E_\Delta}_t = \bigcap_{s > t} \mathcal{D}^{E_\Delta,0}_s\), 
\(\mathfrak{D}^{E_\Delta} = (\mathcal{D}^{E_\Delta}_t)_{t \geq 0}\) and \(\mathfrak{D}^{E_\Delta, 0} = (\mathcal{D}^{E_\Delta, 0}_t)_{t \geq 0}\) on \((\mathbb{D}^{E_\Delta}, \mathcal{D}^{E_\Delta})\).
\begin{remark}
The \(\sigma\)-field \(\mathcal{D}^{E}\) is the topological Borel \(\sigma\)-field on \(\mathbb{D}^{E}\), c.f. \cite{EK}, Proposition 3.7.1.
\end{remark}

We borrow the following definition from \cite{Bichteler02}.
\begin{definition}\label{def full}
We say that a filtered space \((\Omega, \mathcal{F}, \F)\), where \(\F\) is not necessarily right-continuous, is \emph{full}, if the following holds:
\begin{enumerate}
\item[\textup{(i)}]
\(\mathcal{F} = \mathcal{F}_{\infty-} := \bigvee_{t\geq 0} \mathcal{F}_t\).
\item[\textup{(ii)}]
If whenever \((\mathcal{F}_t, \p_t)_{t \geq 0}\) is a \emph{consistent family}, 
i.e. for all \(s < t\) we have 
\begin{align*}
\p_t \big|_{\mathcal{F}_s} = \p_s,
\end{align*}
there exists a probability measure \(\p\) on \((\Omega, \mathcal{F})\) such that for all \(t \geq 0\)
\begin{align*}
\p \big|_{\mathcal{F}_t} = \p_t.
\end{align*}
\end{enumerate}
\end{definition}
\begin{remark}\label{bichteler remark full}
\begin{enumerate}
\item[\textup{(i)}]
Due to \cite{Bichteler02}, p. 167, if \((\Omega, \mathcal{F}, \F)\) is full, then so is \((\Omega, \mathcal{F}, \F^+)\), where 
\[\F^+ := (\mathcal{F}_{t+})_{t \geq 0}
\textup{ and }
\mathcal{F}_{t+} := \bigcap_{s > t} \mathcal{F}_t.\]
\item[\textup{(ii)}]
Proposition 3.9.17 in \cite{Bichteler02} yields that \((\mathbb{D}^E, \mathcal{D}^E, \mathfrak{D}^{E, 0})\) and \((\mathbb{C}^E, \mathcal{C}^E, \mathfrak{C}^{E, 0})\) are full. 
Therefore, due to part (i) of this remark, \((\mathbb{D}^E, \mathcal{D}^E, \mathfrak{D}^{E})\) and \((\mathbb{C}^E, \mathcal{C}^E, \mathfrak{C}^{E})\) are also full. 
\end{enumerate}
\end{remark}

\begin{remark}
We have to impose a small warning concerning the usual conditions. 
Let \((\Omega, \mathcal{F}, \F)\) be full and let \(\p\) be a probability measure on \((\Omega, \mathcal{F}, \F)\).
Then, in general, the completed space \((\Omega, \mathcal{F}^{\hspace{0.03cm} \p}, \F^{\hspace{0.02cm}\p})\) is not full, c.f. \cite{Bichteler02}, Warning 3.9.20.
However, a more careful \textit{enlargement}, called the \textit{natural enlargement} in \cite{Bichteler02}, Definition 1.3.38, has many advantages of the usual completion and preserves the property of fullness, c.f. \cite{Bichteler02}, Proposition 3.9.18 (ii).
\end{remark}

\section{Extension of Probability Measures}\label{Extension of Probability Measures}
In this appendix we shortly discuss extensions of probability measures. 
We call a set \(A \in \mathcal{X}\) an \emph{atom} of a measurable space \((X, \mathcal{X})\), if the relation \(A' \subseteq A,\) \(A' \in \mathcal{X}\) implies \(A' = A\) or \(A' = \emptyset\).
\begin{definition}[\cite{follmer72}, Appendix]
Let \(\mathbb{T}\subseteq \mathbb{R}^+\) be an index set. We call a system \((\Omega, \mathcal{F}_t)_{t \in \mathbb{T}}\) a \emph{standard system}, if the following holds:
\begin{enumerate}
\item[\textup{(i)}]
For all \(t, t' \in \mathbb{T}\) such that \(t \leq t'\) we have \(\mathcal{F}_t \subseteq \mathcal{F}_{t'}\). 
\item[\textup{(ii)}]
For all \(t \in \mathbb{T}\) the measurable space \((\Omega, \mathcal{F}_t)\) is a standard Borel space, c.f. Definition ~V.2.2. in \cite{parthasarathy1967}.
\item[\textup{(iii)}]
For any increasing sequence \((t_n)_{n \in \mathbb{N}}\) in \(\mathbb{T}\), and for any decreasing sequence \((A_n)_{n \in \mathbb{N}}\), where 
 \(A_n\) is an atom of \((\Omega, \mathcal{F}_{t_n})\), we have \(\bigcap_{n \in \mathbb{N}} A_n \not = \emptyset\).
\end{enumerate}
\end{definition}
\begin{remark}
Although one might expect that \((\mathbb{D}^\mathbb{R},\mathcal{D}^{\mathbb{R}, 0}_t)_{t \geq 0}\) and \((\mathbb{C}^\mathbb{R}, \mathcal{C}^{\mathbb{R}, 0}_t)_{t \geq 0}\) are standard systems, this is not the case as the following fact shows:
The function
\begin{align*}
\omega (t) = \sin(1/(t - 1))
\end{align*}
is a continuous real-valued function on \([0, 1)\), but not on \([0,1]\), c.f. \cite{SV} p.17 for a related discussion. 
\end{remark}
\begin{lemma}[\cite{follmer72}, Appendix, \cite{perkowski2015}, Lemma E.1]
The space \((\mathbb{D}^{E_\Delta}, \mathcal{D}^{E_\Delta})\) is a standard Borel space, and for all \(\mathfrak{D}^{E_\Delta}\)-stopping times \(\tau\), the \(\sigma\)-field \(\mathcal{D}^{E_\Delta}_{\tau-}\) is countably generated.
\end{lemma}
\begin{remark}\label{remark b3}
\begin{enumerate}
\item[\textup{(i)}]
Let \(\F = (\mathcal{F}_t)_{t \geq 0}\) be a (not-necessarily right-continuous) filtration on \((\Omega, \mathcal{F})\), such that \((\Omega, \mathcal{F}_t)_{t \geq 0}\) is a standard system,
and let \((\rho_n)_{n \in \mathbb{N}}\) be a sequence of \(\F^+\)-stopping times, then \((\Omega, \mathcal{F}^+_{\rho_n-})_{n \in \mathbb{N}}\) is a standard system, c.f. \cite{follmer72}, Remark 6.1.1.
\item[\textup{(ii)}]
\((\mathbb{D}^{E_\Delta}, \mathcal{D}^{E_\Delta, 0}_t)_{t \geq 0}\) is a standard system, c.f. \cite{follmer72}, Example 6.3.2.
\end{enumerate}
\end{remark}
Standard systems are the key element in Parthasarathy's extension theorem, which reads as follows:
\begin{theorem}[\cite{parthasarathy1967}, Theorem V.4.2]\label{extension P}
Let \((\Omega, \mathcal{F}_t)_{t \in \mathbb{T}}\) be a standard system and denote the measurable space \((\Omega, \bigvee_{t \in \mathbb{T}}\mathcal{F}_t) =: (\Omega, \mathcal{F})\). Then 
for each family \((\mu_t)_{t \in \mathbb{T}}\) of probability measures \(\mu_t\) on \((\Omega, \mathcal{F}_t)\) such that
\begin{align*}
\mu_t(A) = \mu_{s}(A)\ \textit{ for all } A \in \mathcal{F}_{s} \textit{ and }  s \leq t,
\end{align*}
there exists a unique probability measure \(\mu\) on \((\Omega, \mathcal{F})\) such that \(\mu(A) = \mu_t(A)\) for all \(A \in \mathcal{F}_{t}\) and \(t \in \mathbb{T}\).
\end{theorem}
For the discussion in Section \ref{Martingality on Standard Systems} it is convenient that measures on countably generated sub \(\sigma\)-fields have an extension as stated in the following theorem. 
\begin{theorem}[\cite{perkowski2015}, Theorem D.4]\label{extension P2}
Let \((\Omega, \mathcal{F})\) be a standard Borel space and \(\mathcal{G}\subseteq \mathcal{F}\) be countably generated. If \(\mu\) is a measure on \((\Omega, \mathcal{G})\), then there exists a measure \(\bar{\mu}\) on \((\Omega, \mathcal{F})\), such that \(\bar{\mu}|_{\mathcal{G}} = \mu\).
\end{theorem}
Note that \((\Omega, \mathcal{F}, \F)\) may be full, whereas \((\Omega, \mathcal{F}_t)_{t \geq 0}\) is no standard system. 
However, a converse statement holds as we show next.
\begin{lemma}
If \((\Omega, \mathcal{F}_t)_{t \geq 0}\) is a standard system, then \((\Omega, \mathcal{F}_{\infty-},\F)\) is full.
In particular, the filtered spaces \((\mathbb{D}^{E_\Delta}, \mathcal{D}^{E_\Delta}, \mathfrak{D}^{E_\Delta, 0})\) and \((\mathbb{D}^{E_\Delta}, \mathcal{D}^{E_\Delta}, \mathfrak{D}^{E_\Delta})\) are full.
\end{lemma}
\begin{proof}
The first claim is an immediate consequence of Parthasarathy's extension theorem. The second claim follows from the fact that \((\mathbb{D}^{E_\Delta}, \mathcal{D}^{E_\Delta, 0}_t)_{t \geq 0}\) is a standard system, c.f. Remark \ref{remark b3} and
Remark \ref{bichteler remark full}.
\end{proof}
\section{Explicit Conditions for Uniqueness and Exsistence}\label{SMPs and SDEs}
Let \(\mathbb{K}\) and \(\mathbb{H}\) be real separable Hilbert spaces and denote by \(\mathcal{L}(\mathbb{K},\mathbb{H})\) the space of all continuous linear operators from \(\mathbb{K}\) to \(\mathbb{H}\) equipped with the usual norm of bounded operators.
Let \(F : \mathbb{R}^+ \times \mathbb{C}^{\mathbb{H}}_0 \to \mathbb{H}\) and \(G : \mathbb{R}^+ \times \mathbb{C}^{\mathbb{H}}_0\to \mathcal{L}(\mathbb{K}, \mathbb{H})\) be \(\mathfrak{C}^{\mathbb{H}}_0\)-predictable
processes such that \(F(\cdot, 0)\) and \(G(\cdot, 0)\) are constant. We again use the notation \(F(\cdot, \omega) =: F_\cdot(\omega)\) and \(G(\cdot, \omega) =: G_\cdot(\omega)\), and denote the adjoint of \(G\) by \(G^*\).
Moreover, 
let \(Q \in \mathcal{N}(\mathbb{K}, \mathbb{K})\) be non-negative and self-adjoint.
We assume that \(GQ^{1/2}\) is a Hilbert-Schmidt operator.
Define \(\mathcal{T}^*\) as the set of all \(\mathfrak{C}^{\mathbb{H}, 0}_0\)-stopping times \(\rho\) such that for all \(R \in \mathcal{C}^{\mathbb{H}}_{0,\rho-}\) we have \(R \cap \{\rho = 0\} \in \mathcal{C}^{\mathbb{H}, 0}_{0,\rho}\).
Obviously, all positive \(\mathfrak{C}^{\mathbb{H},0}_0\)-stopping times are included in \(\mathcal{T}^*\). 
\begin{proposition}\label{theorem existence uniqueness SMP}
Assume the following: 
\begin{enumerate}
\item[\textup{(i)}]
For all \(\alpha \in (0, \infty)\) there exists a positive \cadlag increasing function \(L^\alpha\) such that
\begin{align*}
\ \qquad\quad \|F(t, \omega) - F(t, \omega^*)\|_\mathbb{H} + \|G(t, \omega) - G(t, \omega^*)\|_{\mathcal{L}(\mathbb{K}, \mathbb{H})} \leq L^{\alpha}_t \sup_{s < t} \|\omega(s) - \omega^*(s)\|_\mathbb{H}
\end{align*}
for all \(t \geq 0\) and all \(\omega, \omega^* \in \{\bar{\omega} \in \mathbb{C}^{\mathbb{H}}_0 : \sup_{s < t} \|\bar{\omega}(s)\|_\mathbb{H} \leq \alpha\}\).
\item[\textup{(ii)}] There exists a constant \(\lambda > 0\) such that
\begin{align*}
\|F(t, \omega)\|^2_\mathbb{H} + \|G(t, \omega)\|^2_{\mathcal{L}(\mathbb{K}, \mathbb{H})} \leq \lambda  \bigg(1 + \sup_{ s < t} \|\omega(s)\|^2_\mathbb{H}\bigg),
\end{align*}
for all \((t, \omega) \in \mathbb{R}^+ \times \mathbb{C}^{\mathbb{H}}_0\).
\end{enumerate}
Then the SMP on \((\mathbb{C}^{\mathbb{H}}_0, \mathcal{C}^{\mathbb{H}}_0, \mathfrak{C}^{\mathbb{H}}_0)\) associated with \((\mathbb{H}; \of 0, \infty\of; \varepsilon_0; \widehat{X}; B, C, 0)\), where 
\begin{align}\label{B,C}
B = F(\widehat{X}) \cdot I,\quad C = G(\widehat{X}) Q G^*(\widehat{X})\cdot I, 
\end{align}
has a solution and satisfies \(\mathcal{T}^*\)-uniqueness. 
\end{proposition}
The proof of this proposition is based on the connection of SMPs and SDEs, which is well-established in the finite-dimensional case, c.f. \cite{J79}.
From now on let \(\rho\) be a \(\mathfrak{C}^{\mathbb{H}, 0}_0\)-stopping time. 
The following definition is in the spirit of \cite{RSZ08}.

\begin{definition}
A triplet \(((\Omega, \mathcal{F}, \F, \p); W; X)\) is called \emph{solution to the SDE associated with \((Q, F, G, \rho)\)}, if \((\Omega, \mathcal{F}, \F, \p)\) is a filtered probability space which satisfies the usual conditions, supporting a \((\mathbb{K}, Q, \F, \p)\)-Brownian motion \(W\) and an \(\F\)-adapted process \(X\) taking values in \(\mathbb{C}^\mathbb{H}_0\), such that for all \(t \geq 0\), \(\p\)-a.s.
\begin{align}\label{SDE def int}
\|F(X)\|_\mathbb{H} \1_{\of 0, \rho(X)\gs}\cdot I_t + \textup{Tr}(G(X)Q G^*(X))\1_{\of 0, \rho(X)\gs} \cdot I_t < \infty, 
\end{align}
and \(\p\)-a.s.
\begin{align}\label{SDE def int rep}
X = F(X)\1_{\of 0, \rho(X)\gs} \cdot I + G(X)\1_{\of 0, \rho(X)\gs} \cdot W. 
\end{align}
We call \((\Omega, \mathcal{F}, \F, \p; W)\) the \emph{driving system} of the solution \(X\).
\end{definition}
Next we define two types of uniqueness, where the first concept is in the spirit of the monograph of \cite{MP80}.
\begin{definition}
\begin{enumerate}
\item[\textup{(i)}] We say that \emph{pathwise uniqueness} holds for the SDE associated with \((Q,F, G, \rho)\), if whenever \(X\) and \(Y\) are two solutions on the same driving system, we have \(X^{\rho(X) \wedge \rho(Y)} = Y^{\rho(X) \wedge \rho(Y)}\) up to indistinguishability. 
\item[\textup{(ii)}] We say that \emph{uniqueness in law} holds for the SDE associated with \((Q,F, G, \rho)\), if whenever \(X\) and \(Y\) are two solutions on possibly different driving systems, the laws, seen as probability measures on \((\mathbb{C}^\mathbb{H}_0, \mathcal{C}^{\mathbb{H}}_0)\), of \(X\) and \(Y\) coincide. 
\end{enumerate}
\end{definition}

\begin{lemma}\label{lemma PU}
The SDE associated with \((Q,F, G, \rho)\) satisfies pathwise uniqueness if and only if all solutions on the same driving system are indistinguishable. 
\end{lemma}
\begin{proof}
The implication \(\Longleftarrow\) is trivial. We prove the implication \(\Longrightarrow\). Let \(X\) and \(Y\) be two solutions on the same driving system.
Since \(\rho\) is a \(\mathfrak{C}^{\mathbb{H}, 0}_0\)-stopping time, Exercise 7.1.21 in \cite{stroock2010} implies that
\begin{align}\label{JS s}
\omega, \omega^* \in \mathbb{C}^\mathbb{H}_0 : \omega_s = \omega_s^*\ \textup{ for all } s \leq t \ \Longrightarrow\ \textup{ for all } A \in \mathcal{C}^{\mathbb{H}, 0}_{0,t},\ \omega \in A \Leftrightarrow\omega^* \in A.
\end{align}
Let \(\omega^* \in \mathbb{C}^\mathbb{H}_0\) such that \(t := \rho(\omega^*) < \infty\).
Since \(\omega^* \in \{\omega \in \mathbb{C}^\mathbb{H}_0 : \rho(\omega) = t\} \in \mathcal{C}^{\mathbb{H},0}_{0,t}\), \eqref{JS s} yields that
\begin{align*}
\omega, \omega^* \in \mathbb{C}^\mathbb{H}_0 : \omega_s = \omega_s^*\ \textup{ for all } s \leq \rho(\omega^*)\ \Longrightarrow\ \rho(\omega) = \rho(\omega^*).
\end{align*}
If \(\rho(\omega^*) = \infty\) the same conclusion holds trivially. 
Now symmetry yields 
\begin{align*}
\omega, \omega^* \in \mathbb{C}^\mathbb{H}_0 : \omega_s = \omega^*_s\ \textup{ for all } s \leq \rho(\omega) \wedge \rho(\omega^*) \ \Longrightarrow\ \rho(\omega) = \rho(\omega^*).
\end{align*}
Hence we obtain that \(X^{\rho(X) \wedge \rho(Y)} = Y^{\rho(X) \wedge \rho(Y)}\) up to indistinguishability yields that a.s. \(\rho(X) = \rho(Y)\), i.e. \(X = Y\) up to indistinguishability. This finishes the proof.
\end{proof}
It is known that in this setting a result in the spirit of \cite{YW} holds, i.e. that pathwise uniqueness implies uniqueness in law, c.f. \cite{RSZ08}. 

It is no surprise that the following relation of SMPs and SDEs, which is well-known to be true in the finite dimensional setting, also holds in the infinite dimensional setting. 
\begin{lemma}\label{coincide law}
Let \(\p\) be a solution to the SMP on \((\mathbb{C}^{\mathbb{H}}_0, \mathcal{C}^{\mathbb{H}}_0, \mathfrak{C}^{\mathbb{H}}_0)\) associated with \((\mathbb{H}; \rho;\) \(\varepsilon_0; \widehat{X}; B, C, 0)\), where \(B, C\) are given as in \eqref{B,C},
then \(\p\) coincides on \(\mathcal{C}^{\mathbb{H}, 0}_{0,\rho}\) with the law of a solution to the SDE associated with \((Q,F, G, \rho)\).
\end{lemma}
\begin{proof}
In view of Proposition \ref{decomp cont SM} we have
\begin{align*}
\widehat{X}_{\cdot \wedge \rho(\widehat{X})}= F(\widehat{X}) \1_{\of 0, \rho(\widehat{X})\gs} \cdot I + \widehat{X}^c_{\cdot \wedge \rho(\widehat{X})},
\end{align*}
where \(\widehat{X}^c_{\cdot \wedge \rho(\widehat{X})}\) is a continuous local \((\mathfrak{C}^\mathbb{H}_0, \p)\)-martingale with 
\begin{align*}\lle\widehat{X}^c_{\cdot \wedge \rho(\widehat{X})}, \widehat{X}^c_{\cdot \wedge \rho(\widehat{X})}\rre^\p = G(\widehat{X}) Q G^*(\widehat{X}) \1_{\of 0, \rho(\widehat{X})\gs} \cdot I.\end{align*}
Hence, a classical representation theorem, c.f. Theorem 8.2 in \cite{deprato}, yields that we can find an extension \((\Omega^* \times \mathbb{C}^\mathbb{H}_0, \mathcal{F}^* \otimes \mathcal{C}^\mathbb{H}_0, \bar{\mathfrak{F}^* \otimes \mathfrak{C}^\mathbb{H}_0}, \p^* \otimes \p =: \bar{\p})\), such that 
\begin{align*}
\widehat{X}_{\cdot \wedge \rho(\widehat{X})}\circ \varphi = F(\widehat{X})\1_{\of 0, \rho(\widehat{X})\gs} \circ \varphi \cdot I +  G(\widehat{X})\1_{\of 0, \rho(\widehat{X})\gs} \circ \varphi \cdot W,
\end{align*}
where \(\varphi : \Omega^* \times \mathbb{C}^\mathbb{H}_0 \to \mathbb{C}^\mathbb{H}_0\) with \(\varphi(\omega^*, \omega) = \omega\), and \(W\) is a \((\mathbb{K}, Q, \bar{\mathfrak{F}^* \otimes \mathfrak{C}^\mathbb{H}_0}, \bar{\p})\)-Brownian motion, denoting by \(\bar{\mathfrak{F}^* \otimes \mathfrak{C}^\mathbb{H}_0}\) the \(\bar{\p}\)-completion of \(\mathfrak{F}^* \otimes \mathfrak{C}^\mathbb{H}_0\).
From Exercise 7.1.21 in \cite{stroock2010} we deduce that 
\begin{align}\label{SDE sol rep}
\widehat{X}_{\cdot \wedge \rho(\widehat{X})} \circ \varphi = F(\widehat{X}_{\cdot \wedge \rho(\widehat{X})} \circ \varphi) \1_{\of 0, \rho(\widehat{X}_{\cdot \wedge \rho(\widehat{X})}\circ \varphi)\gs} \cdot I + G(\widehat{X}_{\cdot \wedge \rho(\widehat{X})}\circ \varphi)  \1_{\of 0, \rho(\widehat{X}_{\cdot \wedge \rho(\widehat{X})}\circ \varphi)\gs} \cdot W.
\end{align}
Moreover, the integrability condition \eqref{SDE def int} is satisfied due to the well-definedness of the \((\mathfrak{C}^\mathbb{H}_0, \p)\)-characteristics of \(\widehat{X}\) and the definition of \(\bar{\p}\). 
Since \(\bar{\p} \circ (\widehat{X}_0 \circ \varphi)^{-1} = \p \circ \widehat{X}^{-1}_0 = \eta\), by definition, we conclude from \eqref{SDE sol rep} that \(\widehat{X}^{\rho(\widehat{X})}\circ \varphi\) is a solution to the SDE associated with \((Q, F, G, \rho)\) on the driving system \((\Omega^* \times \mathbb{C}^\mathbb{H}_0, \mathcal{F}^* \otimes \mathcal{C}^\mathbb{H}_0, \bar{\mathfrak{F}^* \otimes \mathfrak{C}^\mathbb{H}_0}, \bar{\p}; W)\).
Due to Exercise 7.1.21 in \cite{stroock2010} we have for \(G \in \mathcal{C}^{\mathbb{H}, 0}_{0,\rho}\) that
\begin{align*}
\bar{\p} \circ (\widehat{X}^{\rho(\widehat{X})} \circ \varphi)^{-1} (G) =\p \circ \widehat{X}_{\cdot \wedge \rho(\widehat{X})}^{-1} (G) = \p(G).
\end{align*}
This finishes the proof .
\end{proof}
\begin{corollary}\label{coro uniqueness SDE SMP}
Let \(\rho \in \mathcal{T}^*\). If the SDE associated with \((Q, F, G, \rho)\) satisfies uniqueness in law, 
then the SMP on \((\mathbb{C}^{\mathbb{H}}_0, \mathcal{C}^{\mathbb{H}}_0, \mathfrak{C}^{\mathbb{H}}_0)\) associated with \((\mathbb{H}; \rho; \varepsilon_0; \widehat{X}; B, C, 0)\), where \(B, C\) are given by \eqref{B,C}, satisfies uniqueness.
\end{corollary}
\begin{proof}
First note that for all \(R \in \mathcal{C}^{\mathbb{H}}_{0,\rho-}\) we have \(R \cap \{\rho >0\} \in \mathcal{C}^{\mathbb{H}, 0}_{0,\rho}\), c.f. \cite{JS}, III.2.36.
Now the claim follows from the definition of the set \(\mathcal{T}^*\) and Lemma \ref{coincide law}.
\end{proof}
\hspace{-0.6cm}
\textit{Proof of Proposition \ref{theorem existence uniqueness SMP}:}
Due to Theorem 7.2 of \cite{MP80} there exists a solution to the SDE associated with \((Q, F, G, \infty)\). 
Then transporting the solution onto the space \((\mathbb{C}^\mathbb{H}_0, \mathcal{C}^\mathbb{H}_0, \mathfrak{C}^\mathbb{H}_0)\) equipped with the law of the solution as done in the proof of Theorem 14.80 in \cite{J79} yields the existence of the claimed solution to the SMP.
We fix an arbitrary \(\rho \in \mathcal{T}^*\).
Proposition 6.11 in \cite{MP80} yields that the SDE associated with \((Q, F, G, \rho)\) satisfies pathwise uniqueness. 
In view of Lemma \ref{lemma PU}, thanks to the assumption that \(GQ^{1/2}\) is a Hilbert-Schmidt operator, we may apply the classical Yamada-Watanabe-type result given as in Remark 1.10 and Theorem 2.1 in \cite{RSZ08}. This yields that the SDE associated with \((Q, F, G, \rho)\) satisfies uniqueness in law. 
Hence Corollary \ref{coro uniqueness SDE SMP} yields the claimed uniqueness of the SMP.
This finishes the proof.
\qed


\bibliographystyle{plain}
\bibliography{References}

\end{document}